\documentclass[11pt, reqno]{amsart}
\usepackage[utf8]{inputenc}
\usepackage{amssymb}
\usepackage{amsthm}
\usepackage{hyperref}
\usepackage{cleveref}
\usepackage{amsmath}
\usepackage{esvect}
\usepackage[margin=1.0in]{geometry}
\usepackage{comment}
\newtheorem{theorem}{Theorem}[section]
\theoremstyle{definition}

\newtheorem{definition}[theorem]{Definition}
\newtheorem{lemma}[theorem]{Lemma}
\newtheorem{corollary}[theorem]{Corollary}
\newtheorem{proposition}[theorem]{Proposition}
\newtheorem{remark}[theorem]{Remark}
\newtheorem{conjecture}[theorem]{Conjecture}

\numberwithin{equation}{section}

\DeclareMathOperator{\Prob}{Prob}
\DeclareMathOperator{\Aut}{Aut}
\DeclareMathOperator{\Sur}{Sur}
\DeclareMathOperator{\Hom}{Hom}
\DeclareMathOperator{\Cl}{Cl}
\DeclareMathOperator{\Inn}{Inn}
\DeclareMathOperator{\coker}{coker}
\DeclareMathOperator{\rDisc}{rDisc}
\DeclareMathOperator{\Ind}{Ind}
\DeclareMathOperator{\Sel}{Sel}
\DeclareFontFamily{U}{wncy}{}
\DeclareFontShape{U}{wncy}{m}{n}{<->wncyr10}{}
\DeclareSymbolFont{mcy}{U}{wncy}{m}{n}
\DeclareMathSymbol{\cyrBe}{\mathalpha}{mcy}{"42}

\title{Presentations of Galois groups of unramified extensions of global fields and its predicted distribution}
\author{Ken Willyard }
\date{July 2025}
\begin{document}
\begin{abstract}
    Motivated by the work of Liu, we study certain canonical quotients of $G_{\text{\O}}^T(K)$---the Galois group of the maximal unramified extension of a global field $K$ that is split completely at a finite nonempty set of places in $T$---for $\Gamma$-extensions $K/Q$, and prove they have presentations of a particular form. This presentation leads us to the construction of a new random group model as in the work of Liu, Wood, and Zureick-Brown that predicts the distribution of $G_{\text{\O}}^T(K)$ as we vary among $\Gamma$-extensions $K/Q$ with prescribed local conditions at places in $T$, giving a generalization of the non-abelian Cohen-Lenstra-Martinet Heuristics. The key generalization is that $Q$ can be an arbitrary global field, while this comes at a cost of introducing a prime-to-$|\Cl_T(Q)|$ condition in addition to avoiding roots of unity, $|\Gamma|$, and the characteristic if $Q$ is a function field. 
\end{abstract}
\maketitle

\setcounter{tocdepth}{1}
\tableofcontents

\section{Introduction}
 Unramified extensions of number fields have attracted interest of mathematicians for over a century. One way to understand them is through class field theory, which states that for a number field $K$, there is a reciprocity map that provides a natural isomorphism between the class group of $K$, $\Cl(K)$, and the Galois group of the maximal unramified abelian extension of $K$. It follows that the abelianization $G_{\text{\O}}(K)^{\text{ab}}$ is finite, where $G_{\text{\O}}(K)$ is the Galois group of the maximal unramified extension of $K$, $K_{\text{\O}}/K$. However, the group $G_{\text{\O}}(K)$ itself has remained difficult to understand. For instance, Golod and Shafarevich proved that infinite class field towers exist, implying that $G_{\text{\O}}(K)$ could be infinite. Nevertheless, Shafarevich was able to prove that the pro-$\ell$ completion $G_{\text{\O}}(K)(\ell)$ is (topologically) finitely presented in \cite{Sha}, which lets us understand the maximal unramified $\ell$-extension of $K$. 

 In the case where $K/Q$ is a finite Galois extension with Galois group $\Gamma$, the extension $K_{\text{\O}}/Q$ is Galois. If $K_{\text{\O}}'/K$ is the maximal unramified extension of $K$ whose order is coprime to $|\Gamma|$, then the Schur-Zassenhaus theorem states that \[G(K_{\text{\O}}'/Q)\cong G(K_{\text{\O}}'/K)\rtimes \Gamma.\]
 Moreover, $G(K_{\text{\O}}'/K)$ admits a $\Gamma$-action unique up to conjugation, and we can try to understand the presentation of this Galois group or its $\Gamma$-quotients with this added structure. This is natural from the perspective of class groups: $\Cl(K)$ has a natural action of $G(K/Q)\cong\Gamma$, and we want to extract some information about $\Cl(K)$ from this $\Gamma$-action. For example, when $Q=\mathbb{Q}$, the prime-to-$|\Gamma|$ part of $\Cl(K)$ is the quotient of a direct sum of augmentation ideals of $\Gamma$. Liu, Wood, and Zureick-Brown used this perspective in \cite{lwzb} to show that when $K/\mathbb{Q}$ is totally real and $\ell\nmid 2|\Gamma|$, we have a presentation of the form 
 \begin{equation}\label{eq: proell presentation} G_{\text{\O}}(K)(\ell)\cong (\mathcal{F}_n)_\ell/[r^{-1}\gamma(r)]_{r\in X, \gamma\in\Gamma},\end{equation}
 where $(\mathcal{F}_n)_\ell$ is the nonabelian analogue of the augmentation ideal called the free admissible pro-$\ell$ group on $n$ generators (see 3.3 of \cite{lwzb} for a detailed definition), and $X\subset (\mathcal{F}_n)_\ell$ is a set of cardinality $n+1$. The group $[r^{-1}\gamma(r)]$ is the closed normal subgroup topologically generated by elements of the form $r^{-1}\gamma(r)$. We also remark that \cite{lwzb} obtains a presentation of this form for all $n$ large enough. 

 Generalizing the Galois cohomology techniques from the pro-$\ell$ setting, Liu proved a presentation result for the pro-$\mathcal{C}$ completion $G_{\text{\O}}(K)^\mathcal{C}$ for totally real $K$. Here, $\mathcal{C}$ is any finite set of finite $\Gamma$-groups whose orders are coprime to $2|\Gamma|$ in Theorem $1.1$ of \cite{Liu2} (see \Cref{sec: notation} for definition of pro-$\mathcal{C}$ completion and other terms in this paragraph). The work of Liu is significant in part because it suggests that the pro-prime-to-$2|\Gamma|$ completion of $G_{\text{\O}}(K)$---if it is finitely admissibly generated---admits a presentation of the form $\mathcal{F}_n/[r^{-1}\gamma(r)]$ for finitely many $r$, where $\mathcal{F}_n$ is the free admissible pro-prime-to-$2|\Gamma|$ group on $n$ generators. Theorem $1.2$ of \cite{LW2} also provides a presentation result of this form for imaginary $\Gamma$-extensions $K/\mathbb{Q}$ for any set $\mathcal{C}$ so that $G_{\text{\O}}(K)^\mathcal{C}$ is finitely admissibly generated and satisfies all the necessary relative primeness conditions. Before we proceed to the main result, we remark that in all three of \cite{lwzb}, \cite{Liu2}, \cite{LW2}, the case where $Q=\mathbb{F}_q(t)$ and $K$ is a global field is studied as well, and similar theorems are obtained as long as all the groups involved are coprime to the characteristic and we avoid roots of unity of $K$. 
 
 In the first part of this paper, we allow our base field $Q$ to be an arbitrary global field and compute analogous presentations for the pro-$\mathcal{C}$ completion of the Galois group of the maximal unramified extension of $K$ split completely above a finite nonempty set of places $T=\{v_1,\dots, v_{|T|}\}$ of $Q$, $K_{\text{\O}}^T/K$, where $T$ contains all of the archimedean places if $Q$ is a number field. We denote this Galois group by $G_{\text{\O}}^T(K)$. We prove the following theorem: 

 \begin{theorem}\label{thm: main presentation}
     Let $K/Q$ be a Galois extension of global fields with Galois group $\Gamma$, and $T$ be a finite nonempty set of places of $Q$. Let $\Gamma_{v_i}\subset \Gamma$ be a decomposition group of $v_i\in T$ of the extension $K/Q$, and $\mathcal{C}$ be a set of finite $\Gamma$-groups. Suppose all the relative primeness restrictions are satisfied (see \Cref{thm: main} for more details regarding the relative primeness restriction and notation) and $G_{\text{\normalfont\O}}^T(K)^{\mathcal{C}}$ is finitely admissibly generated. Then for $n$ large enough, we have 
 \begin{equation}G_{\text{\normalfont\O}}^T(K)^{\mathcal{C}}\cong_\Gamma \mathcal{F}_n^{\mathcal{C}}/[r_i^{-1}\gamma(r_i),r_j]_{1\leq i\leq n+1, n+2\leq j\leq n+|T|, \gamma\in \Gamma}\end{equation}
    for some choice $r_i\in \mathcal{F}_n^{\mathcal{C}}$ for $1\leq i\leq n$, and $r_i\in(\mathcal{F}_n^{\mathcal{C}})^{\Gamma_{v_{i-n}}}$ for $n+1\leq i \leq n+|T|$.
 \end{theorem}

 Here, $\mathcal{F}_n^\mathcal{C}$ is the pro-$\mathcal{C}$ completion of $\mathcal{F}_n$, and $(\mathcal{F}_n^\mathcal{C})^{\Gamma_{v_{i-n}}}$ is the set of $\Gamma_{v_{i-n}}$-fixed elements of the group $\mathcal{F}_n^\mathcal{C}$. We point out that there is a special relator(s) corresponding to $i=n+1$, where we look at elements of the form $r_{n+1}^{-1}\gamma(r_{n+1})$ for $r_{n+1}\in (\mathcal{F}_n^\mathcal{C})^{\Gamma_{v_{1}}}, \gamma\in \Gamma$. This theorem is proved using results of \cite{Liu2}. In fact, the inequality obtained from these cohomological calculations suggests the presentation in \Cref{thm: main presentation}, simply by reversing the calculations in the proof of \Cref{thm: main} below. Being able to handle arbitrary base fields in many instances gives an indication of the strength of the theory developed in \cite{Liu2}. 
 
 Taking $Q=\mathbb{Q}$ and $\mathcal{C}$ to be the class of all abelian groups prime to the order of $\Delta_\infty:=|\Cl(Q)||\mu(K)||\Gamma|$, \Cref{thm: main presentation} gives a presentation for the prime-to-$\Delta_\infty$-part of class groups as well. Here $\mu(K)$ is the set of roots of unity in $K$. This presentation for class groups can be obtained without using the theory developed in \cite{Liu2}, and we provide a proof of this in \Cref{prop: class group presentation alt}. In fact, when $|\Gamma|=2$, this gives the form of the presentation of the odd part of the class group of quadratic extensions that motivates the random matrix model for the Cohen-Lenstra heuristics introduced in \cite{cl} and outlined in \cite{EV}.

 Let us now turn to the Cohen-Lenstra Heuristics. These wide-open conjectures of Cohen and Lenstra (see \cite{cl}) predict the distribution of class groups of number fields as we vary among certain $\Gamma$-extensions of $\mathbb{Q}$. For example, the Cohen-Lenstra Heuristics predict that the probability that the $\ell$-primary part of the class group of an imaginary quadratic number field, $\Cl(K)[\ell^\infty]$, is isomorphic to a fixed finite abelian $\ell$-group $H$ is inversely proportional to the number of automorphisms of $H$, $|\Aut H|$. Furthermore, they predict that the average number of surjections from these class groups to $H$ is equal to $1$. Here by ``probability" (similar for ``average"), we mean the limit of the discrete probability that an imaginary quadratic extension with discriminant $>-X$ has $\Cl(K)[\ell^\infty]$ isomorphic to $H$, as $X\rightarrow\infty$. This limit of average number of surjections is called the $H$-moment. 

 Cohen and Martinet generalized these conjectures in \cite{cm} to consider distributions of $\Cl(K)$ for arbitrary $\Gamma$-extensions over general base number fields $Q$, and Boston, Bush, and Hajir generalized in a different direction in \cite{bbh}, \cite{bbh2} to address the nonabelian case $G_{\text{\O}}(K)(\ell)$, albeit for quadratic extensions of $\mathbb{Q}$. Analyzing the $H$-moments appearing in both of these generalizations, \cite{lwzb} gave a conjecture generalizing both cases, at least when $Q=\mathbb{Q}$ and $K/\mathbb{Q}$ is a totally real $\Gamma$-extension. In particular, \cite{lwzb} found a random group model whose $H$-moments specialize to the ones found in \cite{cm}, \cite{bbh}, \cite{bbh2}; in fact, the random group model construction was what led \cite{lwzb} to prove the presentation \eqref{eq: proell presentation}, Liu to prove the presentation for the pro-$\mathcal{C}$ quotient in \cite{Liu2} in the totally real case, and Liu and the author to prove the analogous imaginary case in \cite{LW2}. 

 In the second part of this paper, we use the procedure given in \cite{lwzb} and outlined in \cite{LW2} to create a random group model suggested by the presentation in \Cref{thm: main presentation} to conjecture the $H$-moments of the group $G_{\text{\O}}^T(K)$ for $H$ with order coprime to $\Delta_Q:=|\Cl_T(Q)||\mu(Q)||\Gamma|$ for an arbitrary base global field $Q$, where we further require the order of $H$ to be coprime to the characteristic of $Q$ when $Q$ is a function field. Here $\Cl_T(Q)$ is the $T$-class group of $Q$. In particular, let $D$ be a positive integer and $E_{\Gamma,(\Gamma_v)_{v\in T}}(D,Q)$ be the finite set of isomorphism classes of $\Gamma$-extensions $K/Q$ with decomposition group at primes $v\in T$ conjugate to $\Gamma_v$, and the norm of the radical of the discriminant of $K/Q$, denoted $\rDisc K/Q$, satisfying $\rDisc K/Q=D$. Finally, remove a thin subset of extensions as in \Cref{def: E'} from $E_{\Gamma, (\Gamma_v)_{v\in T}}(D,Q)$, and call this subset $E'_{\Gamma, (\Gamma_v)_{v\in T}}(D,Q)$. Then we conjecture the following:
 \begin{conjecture}\label{conj: main intro}
    If $\#E'_{\Gamma,(\Gamma_v)_{v\in T}}(D,Q)\neq0$ for some $D$ and $H$ is a finite admissible $\Gamma$-group whose order is coprime to $\Delta_Q$, then the $H$-moments are given by the formula
    \begin{equation}\lim_{X\rightarrow\infty}\frac{\sum_{D\leq X}\sum_{K\in E_{\Gamma,(\Gamma_v)_{v\in T}}'(D,Q)}|\Sur_\Gamma(G_{\text{\O}}^T(K)_{|\Gamma|'},H)|}{\sum_{D\leq X}\#E_{\Gamma,(\Gamma_v)_{v\in T}}'(D,Q)}=\frac{|H^\Gamma|}{\prod_{v\in T}|H^{\Gamma_v}|}.\end{equation}
\end{conjecture}
 
 See \Cref{conj: main} and the discussion before it for details. Here, $\Sur_\Gamma(G_{\text{\O}}^T(K)_{|\Gamma|'},H)$ denotes the set of $\Gamma$-equivariant surjective homomorphisms from the pro-prime-to-$|\Gamma|$ completion of $G_{\text{\O}}^T(K)$ to $H$. The $\frac{|H^\Gamma|}{\prod_{v\in T}|H^{\Gamma_v}|}$ term in the right hand side of this conjecture is the $H$-moment of the random group model we construct in \Cref{sec: random}, i.e. the average number of surjections from a random group in the random group model to $H$. 

 Sawin and Wood generalize Conjecture $1.3$ of \cite{lwzb} in Conjecture $1.1$ of \cite{SW} and Conjecture $1.3$ of \cite{SW2} by considering arbitrary base fields and allowing roots of unity in the base field to share factors with $|H|$. \Cref{conj: main intro} is a special case of Sawin and Wood's conjectures when $T=S_\infty$ and $E_{\Gamma, (\Gamma_v)_{v\in T}}'(D,Q)$ is replaced by $E_{\Gamma, (\Gamma_v)_{v\in T}}(D,Q)$ under some natural ordering, since we avoid the case when $|H|$ is not coprime to the order of roots of unity in the base field. In this paper, we generalize the random group model construction in \cite{lwzb} to give a random group model that has the same $H$-moments as the conjecture in \cite{SW2}, over arbitrary base fields (avoiding roots of unity as above). For discussion regarding $E'_{\Gamma, (\Gamma_v)_{v\in T}}(D,Q)$ vs $E_{\Gamma, (\Gamma_v)_{v\in T}}(D,Q)$ and the ordering of $\Gamma$-extensions, we refer the reader to the discussion after \Cref{conj: Fq(t) case}, especially \Cref{rmk: malle} and \Cref{rmk: intermediate fields}.

 \Cref{conj: main intro} speculates the $H$-moments of the distribution of $G_{\text{\O}}^T(K)$; we also conjecture the probabilities, as in Conjecture 1.3 of \cite{lwzb} and Conjecture 7.1 of \cite{LW2}. By a direct application of Sawin's result in \cite{saw} on uniqueness of moments in distributions of random pro-$\mathcal{C}$ $\Gamma$-groups for finite $\mathcal{C}$, we obtain that the moment version of \Cref{conj: main intro} for all pro-$\mathcal{C}$ $\Gamma$-groups $H$ implies the probability version for all such $H$ in \Cref{thm: moment implies probability}.  

 It would be interesting to study the function field case of \Cref{conj: main intro} as in \cite{lwzb} and \cite{LW2}. A version of the Cohen-Lenstra-Martinet heuristics for function fields goes back to the conjecture of Friedman and Washington in \cite{FW}, and progress on it was made by Achter, Ellenberg, Landesman, Levy, Lipnowski, Liu, Sawin, Tsimerman, Venkatesh, Westerland, and Wood, through works such as \cite{Ach}, \cite{EVW}, \cite{LT}, \cite{LST}, \cite{lwzb}, \cite{LL}, \cite{LL2} to name a few contributors. Much of the prior work attacks the function field case by using an idea of Ellenberg, Venkatesh, and Westerland to turn this into a problem of counting Frobenius fixed components of Hurwitz spaces, and using a component invariant as in Wood's paper \cite{Woo2}. From there, one computes the limit by comparing the asymptotic count of Frobenius fixed components of Hurwitz spaces corresponding to the groups $\Gamma$ and $H\rtimes\Gamma$, using the theory of Schur covering groups; one needs to understand homological stability of Hurwitz spaces in order to replace the large $q$-limit with large $q$, as in the breakthrough works of Landesman and Levy in \cite{LL} and \cite{LL2}. We look forward to future work that prove versions of \Cref{conj: main intro} for function fields. 

\subsection*{Outline of the paper}
    In \Cref{sec: notation}, we compile notation that will be used throughout this paper. In \Cref{sec: presentation}, we prove \Cref{thm: main presentation} using results of \cite{Liu2}. We also use \Cref{thm: main presentation} to recover the classical random matrix interpretation of the Cohen-Lenstra-Martinet heuristics on class groups. In \Cref{sec: conjectures} we give precise statements of \Cref{conj: main intro} and its variations. In \Cref{sec: random}, we construct the random group model and show this model gives rise to a well-defined countably additive Borel measure $\mu_{\underline{\Gamma}}$ on a certain set of isomorphism classes of profinite groups $\mathcal{P}$ that has the structure of a topological space. We also compute the measure at closed sets constructed as a decreasing intersection of basic open sets. In \Cref{sec: moment}, we compute the $H$-moments of the distribution determined by $\mu_{\underline{\Gamma}}$, which will be the conjectured $H$-moments in \Cref{conj: main intro}. \Cref{sec: random} and \Cref{sec: moment} are direct generalizations of Section $4$-$6$ of \cite{lwzb} and Section $6$ of \cite{LW2} where we allow non-admissible relators. 
 \subsection*{Acknowledgments}
    I would like to thank Yuan Liu for providing helpful comments and encouraging me to write up this paper. I would also like to thank Melanie Matchett Wood for comments and suggesting that we can allow finitely many finite primes to split completely in these results. I am grateful for Cruz Castillo, Peter Koymans, Will Sawin, and Jiuya Wang for comments on an early draft of this paper. I was supported by the UIUC Campus Research Board award RB23063 under Yuan Liu, the GAANN Fellowship, and the Paul T. Bateman Fellowship.

\section{Notation}\label{sec: notation} 
 Throughout this paper, $\Gamma$ will be a finite group. All groups will be profinite unless stated otherwise. For a positive integer $N$, an $N'$-$\Gamma$-group is a profinite group with a continuous $\Gamma$-action where the order of all of its finite quotients is coprime to $N$. Isomorphisms between $\Gamma$-groups will mean $\Gamma$-equivariant isomorphisms. The set $\mathcal{C}$ will denote a set of isomorphism classes of $\Gamma$-groups, and $\overline{\mathcal{C}}$ will denote the smallest set containing $\mathcal{C}$ closed under subquotients and direct products. We stress that we do not allow arbitrary group extensions in the definition of $\overline{\mathcal{C}}$. The pro-$\mathcal{C}$ completion of a $\Gamma$-group $G$, denoted $G^{\mathcal{C}}$, is the maximal quotient of $G$ so that its finite quotients are all in $\overline{\mathcal{C}}$. The group $G$ is level $\mathcal{C}$ or pro-$\mathcal{C}$ if $G^{\mathcal{C}}=G$. The pro-$N'$ completion of $G$ is the maximal quotient of $G$ whose order is coprime to $N$, and will be denoted $G_{N'}$. We use $|G|$ or $\#G$ to denote the supernatural order of the profinite group $G$, which is just the cardinality of $G$ when $G$ is finite. 
     
 The group $F_n(\Gamma)$, abbreviated $F_n$, will denote the free profinite $|\Gamma|'$-$\Gamma$-group on $n$ generators, which is the free profinite $|\Gamma|'$-group on $n|\Gamma|$ generators $x_{i,\gamma}$ for $1\leq i\leq n, \gamma\in\Gamma$, with the action of $\gamma'\in\Gamma$ given by $\gamma'(x_{i,\gamma})=x_{i,\gamma'\gamma}$. The group $\mathcal{F}_n(\Gamma)$, shortened to $\mathcal{F}_n$, is the closed subgroup of $F_n$ topologically generated by elements of the form $x_{i,\gamma}^{-1}\gamma'(x_{i,\gamma})$. We call $\mathcal{F}_n(\Gamma)$ the free admissible $\Gamma$-group on $n$ generators.
 
 The profinite group $\mathcal{F}_n$ is not required to be coprime to $2$ because this is not needed in the construction of the random group model. However, all presentations of Galois groups we consider will be quotients of the pro-$\mathcal{C}$ completion $\mathcal{F}_n^{\mathcal{C}}$ for $\mathcal{C}$ consisting of groups coprime to $2$. 
     
 Similarly, if a $|\Gamma|'$-$\Gamma$-group $G$ is generated by elements of the form $g^{-1}\gamma(g)$, we call $G$ an admissible $\Gamma$-group (see Sections $2$ and $3$ of \cite{lwzb}). If there are finitely many such $g$, then we call $G$ finitely admissibly generated (see Section $5$ of \cite{LW2}). Alternatively, given a $\Gamma$-group $G$, 
     \begin{equation}
         G \text{ is admissible }\Leftrightarrow (|G|, |\Gamma|)=1, \quad G_\Gamma=1,
     \end{equation}
 where we recall that $G^\Gamma$ is the set of elements of $G$ fixed by $\Gamma$, and $G_\Gamma$ is the maximal quotient of $G$ fixed by $\Gamma$. We also define the function 
     \begin{equation}Y:G\rightarrow G^{|\Gamma|} \quad \text{via}\quad  g\mapsto (g^{-1}\gamma(g))_{\gamma\in\Gamma}.\end{equation}

 For a $\Gamma$-group $G$ and subsets $S_1, S_2\subset G$, let $[Y(S_1),S_2]_{G\rtimes\Gamma}$ be the normal $G\rtimes\Gamma$-subgroup of $G$ generated by the coordinates of the elements of $Y(S_1)$ and the elements of $S_2$. When the groups $G, \Gamma$ are clear from context, we omit the subscript and just write $[Y(S_1),S_2]$. If $S_1, S_2$ is a finite ordered tuple (a finite list) of elements of $G$, we use the same notation $[Y(S_1),S_2]$ to denote the normal $G\rtimes \Gamma$-subgroup generated by the coordinates of $Y(s_1)$ and $s_2$, for all $s_i$ appearing in a coordinate of $S_i$.

 In the category of $|\Gamma|'$-$\Gamma$-groups, the functor $G\mapsto G^\Gamma$ is exact. Furthermore, we will extensively use that for a finite group $G$ in this category and any subgroup $\Gamma_0\subset \Gamma$, we have $|Y(G^{\Gamma_0})|=|G^{\Gamma_0}|/|G^\Gamma|$.  
     
 For a global field $K$, $\Cl(K)$ will be the class group of $K$, and $\mu(K)$ will be the group of roots of unity in $K$, while $\mu_\ell$ is the group of $\ell$th roots of unity with the action of the absolute Galois group $G_K$. Moreover, $K^{\text{sep}}$ will be a fixed separable closure of $K$ which contain all of our extensions (if $K/Q$ is a finite extension, then $Q^{\text{sep}}=K^{\text{sep}}$).
    
 For a $\Gamma$-extension $K/Q$, $G_{\text{\O}}^T(K)$ is the Galois group of the maximal unramified extension of $K$ split completely at the places above $T$, $K_{\text{\O}}^T/K$, where $T$ will be a finite set of set of places of $Q$. We sometimes abuse notation and identify $T$ with $T(K)$, the set of places above $T$ in $K$, but this should not cause any confusion. For a number field $Q$, set $r$ and $s$ to be the number of real and complex places of $Q$, respectively. For a finite prime $\mathfrak{p}$ of $\mathbb{Q}$ or $\mathbb{F}_q(t)$, $S_{\mathfrak{p}}$ is the set of places of $Q$ above $\mathfrak{p}$, and $S_{\infty}$ is the set of all archimedean primes of $Q$. Given a prime number $\ell$ coprime to char$(Q)$ and a finite $\mathbb{F}_\ell[G(K_{\text{\O}}/Q)]$-module $A$, define 
     \begin{equation}\delta_{K/Q,\text{\O}}(A):=\dim_{\mathbb{F}_\ell}H^2(G_{\text{\O}}(K),A)^\Gamma-\dim_{\mathbb{F}_\ell}H^1(G_{\text{\O}}(K),A)^\Gamma.\end{equation}
 Also, if $G_T(K)$ denotes the Galois group of the maximal unramified outside $T$ extension $K_T/K$, and $A$ is a finite $G(K_T/Q)$-module coprime to char$(Q)$, then 
     \begin{equation}
         \chi_{K/Q,T}(A):=\frac{\#H^0(G_T(K),A)^\Gamma \#H^2(G_T(K),A)^{\Gamma}}{\#H^1(G_T(K),A)^\Gamma}.
     \end{equation}
     
 Note that by looking at a finite extension of $K$, we can check that the $H^1$ terms are finite so we never have to worry about $\infty-\infty$ or $\infty/\infty$ occurring in these definitions, although the quantities $\delta_{K/Q,\text{\O}}(A), \: \chi_{K/Q, T}(A)$ are potentially infinite. If $A,B$ are both irreducible $\mathbb{F}_\ell[G(K_{\text{\O}}/Q)]$-modules, then $1_B(A)$ is the indicator function that is $1$ when $B\cong A$ as $\mathbb{F}_\ell[G(K_{\text{\O}}/Q)]$-modules and $0$ otherwise. The multiplicity $m(\mathcal{C},n,G,A)$ for a level $\mathcal{C}$ $\Gamma$-group $G$ is the number of isomorphic factors of $A$ appearing in the relations module of a given presentation $\mathcal{F}_n^{\mathcal{C}}\twoheadrightarrow G$, as in the beginning of Section $4$ of \cite{lwzb}.  

 Finally, when given a $\Gamma$-extension of global fields $K/Q$ and a finite nonempty set of places $T$ as above, we set 
     \begin{equation}
         \Delta:=\begin{cases}
         |\Cl_T(Q)||\mu(K)||\Gamma| &\text{if $Q$ is a number field }\\
         \text{char}(Q)|\Cl_T(Q)||\mu(K)||\Gamma|&\text{if $Q$ is a global function field}.\end{cases}
     \end{equation}
 For convenience, we also set
     \begin{equation}
         \Delta_{\infty}:=|\Cl(Q)||\mu(K)||\Gamma|
     \end{equation}
 for number fields, which corresponds to $\Delta$ when $T=S_\infty$. While $\Delta$ depends on $K/Q$ and $T$, we suppress this from the notation for readability. 
     
 See \cite{lwzb}, \cite{Liu2}, \cite{LW2} for more detailed explanations of the notations given here. 

\section{\texorpdfstring{Presentation of $G_{\text{\O}}^T(K)^{\mathcal{C}}$}{Presentation of GOT(K)C}}\label{sec: presentation}
 In this section, we use Liu's results from Section $9$ of \cite{Liu2} to prove upper bounds for the multiplicity of irreducible modules associated to the relations in a pro-$\mathcal{C}$ presentation of $G_{\text{\O}}^T(K)^{\mathcal{C}}$. We obtain presentations in the desired form $[r_i^{-1}\gamma(r_i),r_j]$ by computing the probability that a Haar randomly chosen set of elements $r_i$ normally generates the relations, and showing that this probability is positive precisely when the multiplicities are all no more than the upper bounds we computed. 
 
 The key result is the following analogue of Lemma $11.1$ of \cite{Liu2} and Lemma $6.9$ of \cite{Liu} for pro-$\mathcal{C}$ completions without any restrictions on the primes above infinity. 
 \begin{lemma}\label{lem: main ineq}
    Suppose $Q$ is a number field and $K/Q$ a $\Gamma$-extension. For each place $v$ of $Q$, let $\Gamma_v\subset \Gamma$ be a decomposition group of $v$ in $K/Q$, and $A$ be an odd irreducible $|\Gamma|'$-$G(K_{\text{\O}}/Q)$-module. Then 
    \begin{equation}\delta_{K/Q,{\text{\O}}}(A)\leq \sum_{v\in S_{\infty}(Q)}\dim_{\mathbb{F}_\ell}A^{\Gamma_v}+1_{\mu_\ell}(A)-1_{\mathbb{F}_\ell}(A),\end{equation}
    where $\ell$ is the exponent of $A$. 
 \end{lemma}
 The proof of this lemma rests on the following special case of a proposition of \cite{Liu2}:
 \begin{proposition}[Proposition 9.4 of \cite{Liu2}, $S=\text{\O}$]\label{prop: Liu 9.4}
   Suppose $K/Q$ is a $\Gamma$-extension of number fields, and $A$ be an $\ell$-torsion $|\Gamma|'$-$G(K_{\text{\O}}/Q)$-module for $\ell$ a prime number. Set $T=S_\infty\cup S_\ell$. Then  
    \[\delta_{K/Q,\text{\O}}(A)\leq \log_{\ell}\left(\prod_{v\in S_\infty}\frac{\#\widehat{H}^0(Q_v,A')}{\#H^0(Q_v,A')}\right)+\dim_{\mathbb{F}_\ell}A'^{G(K_T/Q)}-\dim_{\mathbb{F}_\ell}A^{G(K_{\text{\O}}/Q)}+\epsilon_{K/Q,\text{\O}}(A),\]
    where $A'=\Hom(A,\mu_{\ell^\infty})$, the group $\widehat{H}^0(Q_v,A)$ denotes the Tate cohomology of $A$ with respect to $G_{Q_v}$, and $\epsilon_{K/Q,\text{\O}}(A)=-\sum_{v\in S_\ell}\log_\ell|\#A|_v$ for the normalized valuation $|\cdot|_v$. 
 \end{proposition}

 Actually, \cite{Liu2} states this proposition with the product term replaced by $\chi_{K/Q,T}(A)$ and proves a generalization of the global Euler-Poincare characteristic formula (\cite{Liu2}, Proposition $7.1$), but for the purposes of giving a proof sketch, this form is more convenient for us. 

 \begin{proof}[Proof sketch of \Cref{prop: Liu 9.4}]
    See \cite{Liu2} Proposition $9.4$ for the proof. We provide an alternative proof sketch here that is essentially equivalent using the language of Selmer groups. 

    Recall (see for example \cite{MS}, \cite{MS2}) that for a number field $Q$, and a finite $G_Q$-module $A$, local conditions $\mathcal{L}$ consist of a tuple of groups $\mathcal{L}=(\mathcal{L}_v)_v$ for each place $v$ of $Q$ so that $\mathcal{L}_v\subset H^1(Q_v,A)$, and $\mathcal{L}_v=H^1_{\text{ur}}(Q_v,A)$ for all but finitely many places $v$. Here $H^1_{\text{ur}}(Q_v,A):=H^1(G(\widetilde{Q_v}/Q_v),A^{I_v})\subset H^1(Q_v, A)$, the group $I_v$ is the absolute inertia group of $Q_v$, and $\widetilde{Q_v}/Q_v$ is the maximal unramified extension of $Q_v$. Then the Selmer group of $A$ with local conditions $\mathcal{L}$ is defined to be the kernel of the natural map induced by the localization map:
    \begin{equation}
        \Sel_Q(A,\mathcal{L})=\ker\left(H^1(Q,A)\rightarrow \prod_vH^1(Q_v,A)/\mathcal{L}_v\right).
    \end{equation}
    Taking the collection $\mathcal{L}_v=H^1_{\text{ur}}(Q_v, A)$ for all $v$ is called the unramified local conditions. 
    
    For a finite $G(K_{\text{\O}}/Q)$-module $A$, in fact $H^1(G_{\text{\O}}(K),A)=\Sel_K(A,(H^1_{\text{ur}}(K_w,A))_w)$, where $w$ ranges over places of $K$. This follows from comparing the inflation-restriction sequences of global and local cohomology groups and showing that the natural map 
    \[H^1(G_{K_{\text{\O}}},A)^{G_{\text{\O}}(K)}\rightarrow \prod_wH^1(I_w , A)^{G(\widetilde{K_w}/K_w)}\]
    is injective. Moreover, using that $|\Gamma|$ is coprime to the order of $A$, we get 
    \[H^1(G_{\text{\O}}(K),A)^\Gamma=\Sel_K(A, (H^1_{\text{ur}}(K_w,A))_w))^\Gamma=\Sel_Q(A, (H^1_{\text{ur}}(Q_v,A))_v).\]

    Next, given local conditions $\mathcal{L}=(\mathcal{L}_v)_v$, one can define the dual local conditions $\mathcal{L}^\perp:=(\mathcal{L}_v^\perp)_v$ for the $G_Q$-module $A'=\Hom(A,{Q^{\text{sep}}}^\times)$ to be the annihilator of $\mathcal{L}_v\subset H^1(Q_v, A)$ under the local Tate duality pairing. In particular, $\mathcal{L}_v^\perp\subset H^1(Q_v, A')$. Once again, by the coprimality of $|\Gamma|$ and $|A|$, 
    \[\Sel_K(A', (H^1_{\text{ur}}(K_w,A)^\perp)_w))^\Gamma=\Sel_Q(A', (H^1_{\text{ur}}(Q_v,A)^\perp)_v).\]

    Also, we can show that $\#H^2(G_{\text{\O}}(K),A)^\Gamma\leq \#\Sel(A',(H^1_{\text{ur}}(K_w,A)^\perp)_w)^\Gamma$. This essentially follows by comparing the inflation-restriction-transgression sequence with a modified version of the Poitou-Tate long exact sequence where we take into account local conditions (see Proposition $8.5$ of \cite{Liu2}) and using that taking $\Gamma$-invariants is exact on $\Gamma$-modules with $|A|$-torsion because $|\Gamma|$ and $|A|$ are coprime. 

    These facts give
    \begin{equation}
        \frac{\#H^2(G_{\text{\O}}(K),A)^\Gamma}{\#H^1(G_{\text{\O}}(K),A)^\Gamma}\leq \frac{\Sel_Q(A', (H^1_{\text{ur}}(Q_v,A)^\perp)_v)}{\Sel_Q(A, (H^1_{\text{ur}}(Q_v,A))_v)}.
    \end{equation}
    Now by an application of the Wiles formula (see \cite{NSW} Proposition 8.7.9), 
    \begin{equation}
        \frac{\Sel_Q(A', (H^1_{\text{ur}}(Q_v,A)^\perp)_v)}{\Sel_Q(A, (H^1_{\text{ur}}(Q_v,A))_v)}=\frac{\#H^0(Q,A')}{\#H^0(Q,A)}\prod_v\frac{\#H^1_{\text{ur}}(Q_v, A)^\perp}{\#H^0(Q_v, A')}.
    \end{equation}
    When $v|\infty$, $H^1_{\text{ur}}(Q_v,A)=0$, so \[\#H^1_{\text{ur}}(Q_v,A)^\perp=\#H^1(Q_v,A')=\#\widehat{H}^0(Q_v,A').\]
    When $v\nmid \infty$, local Tate duality and the local Euler-Poincare characteristic formula gives
    \begin{equation}
        \frac{\#H^1_{\text{ur}}(Q_v,A)^\perp}{\#H^0(Q_v,A')}=\frac{\#H^1(Q_v,A)}{\#H^1_{\text{ur}}(Q_v,A)\#H^0(Q_v,A')}=\frac{\#H^1(Q_v,A)}{\#H^0(Q_v,A)\#H^2(Q_v,A)}=|\#A|_v^{-1}.
    \end{equation}
    Since the only primes $v$ for which this quantity is nontrivial is $v|\ell$, after taking $\log_\ell$ of both sides, this recovers \Cref{prop: Liu 9.4}.
 \end{proof}
 
 Note that in the proof sketch above, we only needed $A$ to be finite and of order coprime to $|\Gamma|$. In addition, the proof works when $Q$ is a global function field as long as $|\Gamma|$ and $|A|$ are coprime to char$(Q)$ as well; this recovers one direction of Proposition 9.3 in \cite{Liu2}. We also have $\Sel_Q(A', (H^1_{\text{ur}}(Q_v,A)^\perp))=\cyrBe_{\text{\O}}(Q, A)^\vee$, which is the group in Definition $8.1$ of \cite{Liu2}. When $A=\mathbb{F}_\ell$, this coincides with the group $\cyrBe_{\text{\O}}(Q)^\vee$ used in Shafarevich and Koch's proof \cite{Sha}, \cite{Koc} that $G_{\text{\O}}(Q)(\ell)$ is finitely presented. In this case, $\Sel_Q(\mu_\ell, (H^1_{\text{ur}}(Q_v,\mathbb{F}_\ell)^\perp))$ is sometimes referred to as the Selmer group of the number field $Q$, or a Kummer group (see \cite{NSW} Section 9.1). 
 \begin{proof}[Proof of \Cref{lem: main ineq}]
    We compute the right hand side of the inequality in \Cref{prop: Liu 9.4} for odd irreducible $A$. Using the definition of $\epsilon_{K/Q, \text{\O}}(A)$ in \cite{Liu2}, the last term works out to be 
    \[\epsilon_{K/Q,\text{\O}}(A)=-\sum_{v\in S_\ell}\log_\ell|\#A|_v=\sum_{v\in S_\ell}e_vf_v v_\ell(\#A)=[Q:\mathbb{Q}]\dim_{\mathbb{F}_\ell}A.\] 
    The group $A^{G(K_{\text{\O}}/Q)}=0$ unless $A=\mathbb{F}_\ell$, and $A'^{G(K_T/Q)}=0$ unless $A=\mu_\ell$. 
    
    We are left to understand the terms coming from archimedean primes. The group $\widehat{H}^0(Q_v,A')=0$ and $H^0(Q_v,A')=\Hom_{D_v}(A,\mu_\ell)$, where $D_v$ is a decomposition group of $v$ in $\overline{Q}/Q$. If $Q_v=\mathbb{C}$ then $\dim_{\mathbb{F}_\ell}H^0(Q_v,A')=\dim_{\mathbb{F}_\ell}A$. On the other hand if $Q_v=\mathbb{R}$, the action of the decomposition group above $v$ can be identified with the action of $\Gamma_v$, and $A^{\Gamma_v}$ must be sent to $0$ for any $f\in H^0(Q_v,A')$ since $\ell$ is odd. Thus, 
    \[\dim_{\mathbb{F}_\ell}\Hom_{D_v}(A,\mu_\ell)=\dim_{\mathbb{F}_\ell}\Hom(A/A^{\Gamma_v},\mu_\ell)=\dim_{\mathbb{F}_\ell}A/A^{\Gamma_v}.\]
    Therefore, 
    \[\delta_{K/Q,\text{\O}}(A)\leq -s\dim_{\mathbb{F}_\ell}A-\sum_{v\in S_{\infty}}\dim_{\mathbb{F}_\ell}A/A^{\Gamma_v}+1_{\mu_\ell}(A)-1_{\mathbb{F}_\ell}(A)+(r+2s)\dim_{\mathbb{F}_\ell}A,\]
    which gives the lemma. 
 \end{proof}
 
 Now let $Q$ be a global field. Let $T$ be a nonempty finite set of places of $Q$ containing all the archimedean places, and recall that $K_{\text{\O}}^T/K$ is the maximal unramified extension of $K$ split completely at the primes in $T$ and $G_{\text{\O}}^T(K)$ is the Galois group of this extension over $K$. Suppose now that $A$ is an odd finite irreducible $|\Gamma|'$-$G(K_{\text{\O}}^T/Q)$-module whose order is coprime to char$(Q)$ if $Q$ is a function field. Define 
 \begin{equation}\delta_{K/Q,\text{\O}}^T(A):=\dim_{\mathbb{F}_\ell}H^2(G_{\text{\O}}^T(K),A)^\Gamma-\dim_{\mathbb{F}_\ell}H^1(G_{\text{\O}}^T(K),A)^\Gamma.\end{equation}
 \begin{corollary}\label{cor: main ineq function field}
    Under the above set up, for each place $v$ of $Q$ let $\Gamma_v$ be a decomposition group of $v$ in $K/Q$. Then
   \begin{equation} \delta_{K/Q,\text{\O}}^T(A)\leq \sum_{v\in  T}\dim_{\mathbb{F}_\ell}A^{\Gamma_v}+1_{\mu_\ell}(A)-1_{\mathbb{F}_\ell}(A)\end{equation}
   except when $K$ is a function field of genus $0$, in which case we remove the $1_{\mu_\ell}(A)$ term. 
 \end{corollary}
 \begin{proof}
 Using the inflation-restriction-transgression sequence for the module $A$ and the group extension $G(K_{\text{\O}}/Q)\rightarrow G(K_{\text{\O}}^T/Q)$, we see that 
 \[\delta_{K/Q,\text{\O}}^T(A)\leq \delta_{K/Q,\text{\O}}(A)+\dim_{\mathbb{F}_\ell}H^1(G(K_{\text{\O}}/K_{\text{\O}}^T),A)^{G(K_{\text{\O}}^T/Q)}.\]
 In this case, we can identify the second $H^1$ on the right-hand side with $\Hom$. Such homomorphisms factor through the $|\Gamma|'$-quotient of $G(K_{\text{\O}}/K_{\text{\O}}^T)$ because $|A|$ is coprime to $|\Gamma|$. Using that the cyclic decomposition groups of the primes lying above $T$ generate $G(K_{\text{\O}}/K_{\text{\O}}^T)$ and that conjugation by elements of $G(K_{\text{\O}}^T/Q)$ transitively acts on the primes of $K_{\text{\O}}$ above any fixed prime of $T$, we see that 
 \[\dim_{\mathbb{F}_\ell}H^1(G(K_{\text{\O}}/K_{\text{\O}}^T)_{|\Gamma|'},A)^{G(K_{\text{\O}}^T/Q)}\leq \sum_{v\in T\backslash S_{\infty}}\dim_{\mathbb{F}_\ell}\Hom(G((K_{\text{\O}})_{\overline{v}}/K_{w})_{|\Gamma|'},A)^{\Gamma_v}\]
 where $\overline{v},w,v$ are places of the fields $K_{\text{\O}}, K,Q$, respectively, all chosen compatibly with respect to the choice of the decomposition group $\Gamma_v$ as well. Also, $(K_{\text{\O}})_{\overline{v}}, K_{w},Q_v$ are completions of the respective fields at $\overline{v}, w,v$, respectively. Also recall that the subscript $|\Gamma|'$ under each of the Galois groups indicates the pro-$|\Gamma|'$ completion of the respective Galois groups, and $(K_{\text{\O}})_{\overline{v}}'$ is the maximal prime-to-$|\Gamma|$ unramified subextension of $(K_{\text{\O}})_{\overline{v}}/K_{w}$. The extension $(K_{\text{\O}})_{\overline{v}}'/Q_v$ is Galois and the natural map \[G((K_{\text{\O}})_{\overline{v}}'/K_w)\hookrightarrow G((K_{\text{\O}})_{\overline{v}}'/Q_v)\twoheadrightarrow G(Q_v^{\text{nr}}/Q_v)\]
 is injective, where $Q_v^{\text{nr}}$ is the maximal unramified subextension of $(K_{\text{\O}})_{\overline{v}}'/Q_v$. 
 
 Both groups $G(Q_v^{\text{nr}}/Q_v)$ and $G((K_{\text{\O}})_{\overline{v}}'/K_w)$ are Galois groups of unramified extensions, so they are procyclic. This is enough to guarantee that the $\Gamma_v$-action on $G((K_{\text{\O}})_{\overline{v}}/K_{w})_{|\Gamma|'}$ is trivial, so 
 \[\Hom(G((K_{\text{\O}})_{\overline{v}}/K_{w})_{|\Gamma|'},A)^{\Gamma_v}=\Hom(G((K_{\text{\O}})_{\overline{v}}/K_{w})_{|\Gamma|'},A^{\Gamma_v}),\]
 and hence 
 \[\dim_{\mathbb{F}_\ell}\Hom(G((K_{\text{\O}})_{\overline{v}}/K_{w})_{|\Gamma|'},A)^{\Gamma_v}\leq \dim_{\mathbb{F}_\ell}A^{\Gamma_v}.\]
 The corollary follows after using \Cref{lem: main ineq} and Proposition $9.3$ of \cite{Liu2}. 
 \end{proof}
 
 Recall that $\Cl_T(Q)$ is the $T$-class group of $Q$, which is the class group of $Q$ quotiented by all ideal classes represented by finite primes in $T$. By class field theory, this corresponds to the maximal unramified abelian extension of $Q$ split completely at the primes above $T$. As long as $T$ is nonempty, $\Cl_T(Q)$ is a finite group. We now prove an analogue of Proposition $2.2$ in \cite{lwzb}. 
\begin{proposition}[Analogue of Proposition $2.2$ of \cite{lwzb}]\label{prop: admissibility}
    Let $K/Q$ be a $\Gamma$-extension of global fields, and $T$ be a nonempty set of primes containing all archimedean places. Set $N=|G_{\text{\O}}^T(Q)||\Gamma|$, and let $G_{\text{\O}}^T(K)_{N'}$ be the pro-prime-to-$N$ completion of $G_{\text{\O}}^T(K)$, that is the maximal quotient of $G_{\text{\O}}^T(K)$ that is not divisible by any prime dividing $|\Gamma|$ or the supernatural order of $G_{\text{\O}}^T(Q)$. Then the group $G_{\text{\O}}^T(K)_{N'}$ and all its $\Gamma$-quotients are admissible. 
\end{proposition}
\begin{proof}
    The same proof as Proposition $2.2$ of \cite{lwzb} works. Indeed if $L/K$ denotes the extension with Galois group $G=G_{\text{\O}}^T(K)_{N'}$, then $L/Q$ is Galois with Galois group $G\rtimes\Gamma$. The quotient by the closed normal $\Gamma$-group $N\subset G$ generated by all elements of the form $g^{-1}\gamma(g)$ for $\gamma\in\Gamma, g\in G$ has trivial $\Gamma$-action, so $(G\rtimes\Gamma)/N\cong \overline{G}\times \Gamma$. In other words, $G(L^N/Q)\cong \overline{G}\times \Gamma$. Taking $\Gamma$-fixed points of $L^N$, we see that $L^{N\rtimes \Gamma}/Q$ is a $\overline{G}$-extension. Since all inertia groups divide the order of $|\Gamma|$ and $G$ is relatively prime to $|\Gamma|$, the extension $L^{N\rtimes \Gamma}/Q$ is unramified. All decomposition groups above primes in $T$ also divide $|\Gamma|$, so $L^{N\rtimes\Gamma}/Q$ is also completely split at those primes. Since we made sure that $|G|$ is relatively prime to $|G_{\text{\O}}^T(Q)|$ as well, $\overline{G}=1$, which proves the admissibility of $G$. 
\end{proof}

 Note that we have little control on the group $G_{\text{\O}}^T(Q)$. For example, we do not know if this group has supernatural order divisible by only finitely many primes or not. However, if the quotient of $G_{\text{\O}}^T(K)_{N'}$ we were considering was solvable, then at the end of the proof above, we only needed $|G|$ to be relatively prime to the order of the finite group $\Cl_T(Q)$ in order to conclude that $\overline{G}=1$. This gives us a convenient corollary of the proof:
 \begin{corollary}\label{cor: admissibility}
    The maximal pro-solvable $\Gamma$-quotient of $G_{\text{\O}}^T(K)$ that has order prime to $|\Cl_T(Q)||\Gamma|$ is admissible. 
 \end{corollary}
 Since pro-odd groups are pro-solvable, we can apply the corollary above to quotients of $G_{\text{\O}}^T(K)$ that are $(2|\Cl_T(Q)||\Gamma|)'$-$\Gamma$ groups. We can now state the main upper bound on the multiplicity of admissible presentations. 

 \begin{corollary}\label{cor: multiplicity bound}
 Under the setup of \Cref{cor: main ineq function field}, suppose also that $\mathcal{C}$ is a set of $\left(2|\Cl_T(Q)||\Gamma|\right)'$-$\Gamma$ groups such that $G_{\text{\O}}^T(K)^{\mathcal{C}}$ is finitely admissibly generated. In addition, assume that $A$ is an odd irreducible $|\Gamma|'$-$G_{\text{\O}}^T(K)^{\mathcal{C}}\rtimes\Gamma$-module. Then we have
 \begin{equation}m(\mathcal{C},n,G_{\text{\O}}^T(K)^\mathcal{C},A)\leq \frac{n\dim_{\mathbb{F}_\ell}A-(n+1)\dim_{\mathbb{F}_\ell}A^\Gamma+\sum_{v\in  T}A^{\Gamma_v}+1_{\mu_\ell}(A)}{\dim_{\mathbb{F}_\ell}\Hom_{G_{\text{\O}}^T(K)^\mathcal{C}\rtimes\Gamma}(A,A)}.\end{equation}
 \end{corollary}
 \begin{proof}
    The proof proceeds in the exact same way as Sections 10.1-10.2 of \cite{Liu2} and Proposition $5.1$ of \cite{LW2}. We give a sketch.
    
    Indeed, following the argument in Proposition $5.1$, we can construct a $\Gamma$-group extension $G\twoheadrightarrow G_{\text{\O}}^T(K)^\mathcal{C}$ that has the following properties:
    \begin{enumerate}
        \item $G$ is a $\Gamma$-quotient of $G_{\text{\O}}^T(K)$ such that $G_{\text{\O}}^T(K)^\mathcal{C} = G^{\mathcal{C}}$.
        \item A subset of $G$ generates $G$ if and only if its image in $G_{\text{\O}}^T(K)^\mathcal{C}$ generates $G_{\text{\O}}^T(K)^\mathcal{C}$.
        \item $\dim_{\mathbb{F}_\ell}H^2(G,A)^\Gamma-\dim_{\mathbb{F}_\ell}H^1(G,A)^\Gamma\leq \delta_{K/Q,\text{\O}}^T(A)$.
    \end{enumerate}
    By $(2)$, the $\Gamma$-group $G$ is finitely admissibly generated. Therefore, we can construct a surjection $F_n\twoheadrightarrow G$ that restricts to a surjection $\mathcal{F}_n\twoheadrightarrow G$, so that pro-$\mathcal{C}$ completion gives 
    \[\mathcal{F}_n^\mathcal{C}\twoheadrightarrow G^\mathcal{C}=G_{\text{\O}}^T(K)^\mathcal{C},\]
    where we use $(1)$ in the last equality. Now use Proposition $3.4$, Corollary $4.5$ and Proposition $5.4$ of \cite{Liu2}, as well as $(3)$ and \Cref{cor: main ineq function field} to obtain the desired inequality.  
 \end{proof}

 Note that we cannot drop the hypothesis on being relatively prime to $|\Cl_T(Q)||\Gamma|$ since this is what guarantees admissibility of the group $G_{\text{\O}}^T(K)^{\mathcal{C}}$, but we can weaken the conditions of the corollary on groups in $\mathcal{C}$ to being solvable and order prime to $|\Cl_T(Q)||\Gamma|$. Theorem $6.4$ of \cite{Liu2} states that for all finite $\mathcal{C}$, $G_{\text{\O}}^T(K)^{\mathcal{C}}$ is finite, hence finitely admissibly generated if we impose the relative primeness condition on groups in $\mathcal{C}$. Other choices of $\mathcal{C}$ we can take so that $G_{\text{\O}}^T(K)$ is finitely admissibly generated are the set of all $p$ groups for some fixed prime $p\nmid |\Cl_T(Q)||\Gamma|$, the set of all nilpotent groups or the set of all abelian groups prime to $|\Cl_T(Q)||\Gamma|$, among others. 

 Next we provide a slight generalization of Proposition $4.3$ of \cite{lwzb} and Lemma $5.3$ of \cite{LW}. In particular, suppose we have $u+1$ subgroups $\Gamma_1,\Gamma_2,\dots, \Gamma_{u+1}\subset \Gamma$. Also, given a $\Gamma$-group $F$ and a finite $F\rtimes\Gamma$-module $G$, we denote
 \begin{equation}h_{F\rtimes\Gamma}(G):=\#\Hom_{F\rtimes\Gamma}(G,G).\end{equation}
 \begin{proposition}\label{prop: relation probability}
     Suppose that $1\rightarrow R\rightarrow F\rightarrow H \rightarrow 1$ is an exact sequence of profinite $|\Gamma|'$-$\Gamma$-groups such that $\displaystyle R\cong_{F\rtimes\Gamma}\prod_{j\in J}G_j^{m_j}$, where $G_j$ are finite irreducible $F\rtimes\Gamma$-groups such that for $j\neq j'$, $G_j\not\cong G_{j'}$ as $F\rtimes\Gamma$-groups, $m_j$ are finite, and $J$ is some indexing set. If $(r_1, \dots , r_{n+u+1})$ is chosen Haar randomly and uniformly from $R^n\times R^{\Gamma_1}\times R^{\Gamma_2}\times\cdots R^{\Gamma_{u+1}}$ for $n\geq 0$, we have 
     \begin{align}&\Prob ([Y(\{r_1,\dots,r_{n+1}\}),r_{n+2},\dots, r_{n+u+1}]_{F\rtimes\Gamma}=R)\nonumber \\
     &=\prod_{G_j\: \text{abelian}}\prod_{\ell=0}^{m_j-1}\left(1-\frac{h_{F\rtimes\Gamma}(G_j)^\ell|G_j^\Gamma|}{|Y(G_j)|^n\prod_{k=1}^{u+1}|G_j^{\Gamma_k}|}\right)\prod_{G_j \: \text{nonabelian}}\left(1-\frac{|G_j^\Gamma|}{|Y(G_j)|^n\prod_{k=1}^{u+1}|G_j^{\Gamma_k}|}\right)^{m_j}.
     \end{align}
     Moreover, there exists a choice of $(r_i)$ such that $[Y(\{r_1,\dots, r_{n+1}\}),\dots, r_{n+u+1}]_{F\rtimes\Gamma}=R$ if and only if the factors $\left(1-\frac{h_{F\rtimes\Gamma}(G_j)^\ell|G_j^\Gamma|}{|Y(G_j)|^n\prod_{k=1}^{u+1}|G_j^{\Gamma_k}|}\right), \: \left(1-\frac{|G_j^\Gamma|}{|Y(G_j)|^n\prod_{k=1}^{u+1}|G_j^{\Gamma_k}|}\right)^{m_j}$ appearing in the product are all nonzero. 
 \end{proposition}
 \begin{proof}
    The proof proceeds in the exact same way as the proof of Proposition $4.3$ of \cite{lwzb} and Lemma $5.3$ of \cite{LW2}. We sketch the argument and record the minor differences that arise. 

    First suppose $|J|<\infty$. Indeed by Lemma $5.5$ of \cite{LW}, 
    \begin{align*}&\Prob([Y(\{r_1,\dots,r_{n+1}\}),r_{n+2},\dots, r_{n+u+1}]_{F\rtimes\Gamma}=R)\\
    &=\prod_{j\in J}\Prob([Y(\{g_{ij}\}_{1\leq i \leq n+1}),\{g_{ij}\}_{n+2\leq i\leq n+u+1}]_{F\rtimes\Gamma}=G_j^{m_j}).\end{align*}
    Here $g_{ij}$ is the $G_j^{m_j}$ coordinate of $r_i\in R\cong\prod_{j\in J}G_j^{m_j}$. This reduces us to the case when $R=G^m$, where $G$ is a finite irreducible $F\rtimes\Gamma$-group and $m$ is a positive integer. 

    In the nonabelian case, the probability that $[Y(\{r_1,\dots,r_{n+1}\}),r_{n+2},\dots, r_{n+u+1}]_{F\rtimes\Gamma}=G^m$ is the probability that the projection to each coordinate is nontrivial by Lemma $5.8$ of \cite{LW}, which is  
    \[\prod_{\ell=1}^m\Prob(Y(r_{i\ell})\neq 1 \text{ for  some } 1\leq i\leq n+1, \text{ or }r_{i\ell}\neq 1\text{ for some }n+2\leq i\leq n+u+1).\]
    Here $r_{i\ell}$ is the $\ell$th coordinate of $r_i\in G^m$. The probability that $Y(r_{i\ell})=1$ is $\frac{1}{|Y(G)|}$ for $1\leq i \leq n$ and $\frac{1}{|Y(G^{\Gamma_1})|}=\frac{|G^\Gamma|}{|G^{\Gamma_1}|}$ for $i=n+1$. The probability that $r_{i\ell}=1$ is $\frac{1}{|G^{\Gamma_{i-n}}|}$ for $n+2\leq i \leq n+u+1$. Since these events are independent, the above probability is 
    \[\prod_{\ell=1}^m\left(1-\frac{|G^{\Gamma}|}{|Y(G)|^n\prod_{k=1}^{u+1}|G^{\Gamma_k}|}\right)=\left(1-\frac{|G^{\Gamma}|}{|Y(G)|^n\prod_{k=1}^{u+1}|G^{\Gamma_k}|}\right)^m.\]

    Similarly, in the abelian case, by Lemma $5.8$ of \cite{LW}, a set of elements generate $G^m$ as an $F\rtimes\Gamma$-group if and only if the projection onto the first $m-1$ factors is surjective and the projection onto the last factor does not factor through the projection onto the first $m-1$ factors. Our approach is to iteratively count the complement.
    
    Let $\pi$ be the projection onto the first $m-1$ factors. Suppose that the projection of $r_i$ onto the first $m-1$ factors has been chosen already so that the coordinates of $Y(\pi(r_i))$ for $1\leq i \leq n+1$ and $\pi(r_i)$ for $n+2\leq i \leq n+u+1$ generate $G^{m-1}$. For each map $f\in\Hom_{F\rtimes\Gamma}(G^{m-1},G)$, specifying $f$ amounts to specifying what it does on coordinates of $Y(\pi(r_i))$ for $1\leq i \leq n+1$ and $\pi(r_i)$ for $n+2\leq i \leq n+u+1$ by the generating hypothesis. Since $f$ is $\Gamma$-equivariant, we see that if $\pi(r_i)$ were fixed by $\Gamma_k$ then its image under $f$ would be fixed under $\Gamma_k$. If we let $f$ assign the last coordinate of $r_i$, by Lemma $5.8$ of \cite{LW}, the set of coordinates of $Y(r_i)$ for $1\leq i \leq n+1$ and $r_i$ for $n+2\leq i \leq n+u+1$ does not generate $G^m$. 
    
    On the other hand, if we started with a non-generating set consisting of coordinates of $Y(r_i)$ for $1\leq i \leq n+1$ projecting to the chosen $Y(\pi(r_i))$, and $r_i$ for $n+2\leq i \leq n+u+1$ projecting to the chosen $\pi(r_i)$, then by Lemma $5.8$ of \cite{LW} again, we would be able to find $f$ so that the projection to the last coordinate factors as $f\circ\pi$, and such $f$ is unique. Thus, the number of non-generating tuples $(Y(r_1),\dots, Y(r_{n+1}),r_{n+2},\dots, r_{n+u+1})$ with fixed projection to the first $m-1$ coordinates is $\#\Hom_{F\rtimes\Gamma}(G^{m-1},G)=h_{F\rtimes\Gamma}(G)^{m-1}$. The number of possible tuples is $|Y(G)|^n|Y(G^{\Gamma_1})|\prod_{k=2}^{u+1}|G^{\Gamma_k}|$, so the probability that such a tuple generates $R=G^m$ is 
    \[\left(1-\frac{h_{F\rtimes\Gamma}(G)^{m-1}}{|Y(G)|^n|Y(G^{\Gamma_1})|\prod_{k=2}^{u+1}|G^{\Gamma_k}|}\right).\]
    Multiplying all the conditional probabilities together gives 
    \[\Prob ([Y(\{r_1,\dots,r_{n+1}\}),r_{n+2},\dots, r_{n+u+1}]_{F\rtimes\Gamma}=G^m)=\prod_{\ell=0}^{m-1}\left(1-\frac{h_{F\rtimes\Gamma}(G)^{\ell}|G^\Gamma|}{|Y(G)|^n\prod_{k=1}^{u+1}|G^{\Gamma_k}|}\right).\]
    
    In the case when $|J|$ is infinite, $[Y(\{r_1,\dots,r_{n+1}\}),r_{n+2},\dots, r_{n+u+1}]_{F\rtimes\Gamma}=R$ if and only if the image of $[Y(\{r_1,\dots,r_{n+1}\}),r_{n+2},\dots, r_{n+u+1}]_{F\rtimes\Gamma}$ in the finite set $R_{J'}=\prod_{j\in J'}G_j^{m_j}$ of $R$, generates $R_{J'}$ for all finite subsets $J'\subset J$. This follows from looking at the product topology on $R$. Denote by $R(J')\subset R^n\times R^{\Gamma_1}\times\cdots\times R^{\Gamma_{u+1}}$ the set of tuples $(r_i)$ whose coordinate-wise image $(\overline{r_i})$ in $R_{J'}$ satisfies $[Y(\{\overline{r_1},\dots, \overline{r_{n+1}}\}),\dots, \overline{r_{n+u+1}}]_{F\rtimes\Gamma}=R_{J'}$. Note that $R(J')$ is a closed subset because it is the preimage of a subset of $R_{J'}^n\times\cdots R_{J'}^{\Gamma_{u+1}}$, which is finite and discrete. Then the probability we are interested in computing in this proposition is 
    \begin{equation}
        \Prob\left((r_i)\in \bigcap_{J'\subset J}R(J')\right).
    \end{equation}
    
    In the case when $J$ is countable, we can enumerate elements of $J$, and only take $J'=J_n$ in the intersection above, where $J_n$ is the set consisting of the first $n$ elements of $J$. Since $J_n$ is monotone in $n$, the probability above is equal to $\lim_{n\rightarrow \infty}\Prob((r_i)\in R(J_n))$, which is the desired equality. 

    When $J$ is uncountable, we can find a positive integer $N$ where the set of $j\in J$ such that $\Prob((r_i)\in R(\{j\}))<1-\frac{1}{N}$ is infinite, because all such probabilities are less than $1$. Set $J_n\subset J$ to be a monotonically increasing family of sets so that $J_n$ consists of $n$ elements satisfying this inequality. Then 
    \[\Prob\left((r_i)\in \bigcap_{J'\subset J}R(J')\right)\leq\lim_{n\rightarrow \infty}\Prob\left((r_i)\in R(J_n)\right)\leq \lim_{n\rightarrow \infty}\left(1-\frac{1}{N}\right)^n.\]
    The last limit is $0$, so this proves the desired equality.

    For the second statement, the existence of such a choice of a tuple $(r_i)$ would force each of the factors in the product to be nonzero by looking at their image in the finite group $R_{J'}$ for all $J'$. For the reverse direction, every finite intersection of $R(J')$ for different $J'$ is nonempty by the terms in the product being nonzero. By compactness, $\cap_{J'} R(J')\neq \text{\O}$, which finishes the proof. 
 \end{proof}
 
 If $G_i$ was a profinite $F\rtimes\Gamma$-group, then for any open normal subgroup $U\subset G_i$, the intersection $\bigcap_{f\in F\rtimes\Gamma}f(U)$ is also an open normal subgroup by the continuity of the $F\rtimes\Gamma$-action and profiniteness of the groups $F\rtimes\Gamma, G_i$. Thus, irreducible $F\rtimes\Gamma$-groups are automatically finite. 
 \begin{corollary}\label{cor: fixed point number}
    For each irreducible abelian $H\rtimes\Gamma$-group $G$ and a subgroup $\Gamma_0\subset\Gamma$, there are unique nonnegative integers $m(\Gamma_0), n(\Gamma_0)$ so that 
    \begin{equation}h_{H\rtimes\Gamma}(G)^{m(\Gamma_0)}=|G^{\Gamma_0}|, \quad h_{H\rtimes\Gamma}(G)^{n(\Gamma_0)}=|Y(G^{\Gamma_0})|.\end{equation}
 \end{corollary}
 \begin{proof}
    For the first part, take $\Gamma_1=\Gamma, \Gamma_2=\Gamma_0, n=0, u=1$ in the previous proposition. Since $h_{H\rtimes\Gamma}(G)>1$, via \Cref{prop: relation probability}, the minimal $m$ for which $h_{H\rtimes\Gamma}(G)^m\geq |G^{\Gamma_2}|$ is when we get equality. For the second part, take $\Gamma_1=\Gamma_0, n=0, u=0$ instead. 
 \end{proof}

 In \Cref{prop: relation probability}, we prove more than what is necessary for the main theorem by considering the possibility that the $G_i$ could be nonabelian irreducible $F\rtimes\Gamma$-groups. We never see this in our application to the proof of the main theorem because we are only looking at pro-odd groups. Indeed, the commutator $[G_i,G_i]$ would also be an $F\rtimes\Gamma$-group, and if $G_i$ were odd, it would be solvable, so $[G_i,G_i]\neq G_i$. This forces $[G_i,G_i]=1$ and $G_i$ to be abelian via irreducibility. The pro-odd restriction comes from the condition that we avoid the order of roots of unity (to the characteristic if char$(Q)=2$) since $\pm 1\in Q$. Nevertheless, the random group model we construct in the second half of this paper consists of random groups that do not necessarily have to be pro-odd if $2\nmid |\Gamma|$.

 We now turn to the main theorem. For convenience, we recall here that $\Delta=|\Cl_T(Q)||\mu(K)||\Gamma|$ if $Q$ is a number field and $\Delta=\text{char}(Q)|\Cl_T(Q)||\mu(K)||\Gamma|$ if $Q$ is a function field.
 \begin{theorem}\label{thm: main}
    Let $Q$ be a global field and $K/Q$ be a $\Gamma$-extension. Let $T$ be a nonempty finite set of places of $Q$ containing all the archimedean places if $Q$ is a number field, and $\Gamma_{v_i}$ be a decomposition group of $v_i\in T$ of the extension $K/Q$, where we label the places in $T$ by $v_1,\dots, v_{|T|}$. Let $\mathcal{C}$ be a set of $\Delta'$-$\Gamma$-groups, with the additional condition that $G_{\text{\normalfont\O}}^T(K)^\mathcal{C}$ is finitely admissibly generated. Then for $n$ large enough, we have the following isomorphism of $\Gamma$-groups:
    \begin{equation}G_{\text{\normalfont\O}}^T(K)^{\mathcal{C}}\cong \mathcal{F}_n^{\mathcal{C}}/[r_i^{-1}\gamma(r_i),r_j]_{1\leq i\leq n+1, n+2\leq j\leq n+|T|, \gamma\in \Gamma}\end{equation}
    for some choice $r_i\in \mathcal{F}_n^{\mathcal{C}}$ for $1\leq i\leq n$, $r_i\in(\mathcal{F}_n^{\mathcal{C}})^{\Gamma_{v_{i-n}}}$ for $n+1\leq i \leq n+|T|$.
 \end{theorem}
 \begin{remark}\label{rmk: main thm}
    Reordering the primes gives potentially different presentations, in the sense that the $n+1$st admissible relator may come from a different group $({\mathcal{F}_n^{\mathcal{C}}})^{\Gamma_{v_1}}$. Also, the minimal $n$ that works in the theorem above is one where $G_{\text{\O}}^T(K)^\mathcal{C}$ is finitely admissibly generated by $n$ elements. If we drop the hypothesis on roots of unity, the multiplicity bound we obtain from the group presentation would be too small in comparison with that given by Galois cohomology. In particular, the multiplicity bound coming from Galois cohomology has an extra $1_{\mu_\ell}(A)$ term in the numerator for the multiplicity bound in \Cref{cor: multiplicity bound}, which contributes nontrivially when $A\cong_{G(Q^{\text{sep}}/Q)}\mu_\ell$. Note that avoiding roots of unity (or the characteristic if char$(Q)=2$) forces $G_{\text{\O}}^T(K)^{\mathcal{C}}$ to be solvable, hence admissible by \Cref{cor: admissibility}. 
 \end{remark}
 \begin{remark}
    We also remark that this theorem depends on the decomposition group above the places of $T$, but does not depend on the inertia group above these places. 
 \end{remark}

 \begin{proof}[Proof of \Cref{thm: main}]
    Take $n$ large enough so there is a surjection $\phi:\mathcal{F}_n^\mathcal{C}\rightarrow G_{\text{\O}}^T(K)^\mathcal{C}$. Recall from Section $4$ of \cite{lwzb} that associated to a surjection $\phi: \mathcal{F}_n^\mathcal{C}\rightarrow G_{\text{\O}}^T(K)^\mathcal{C}$, we have a fundamental exact sequence $1\rightarrow R\rightarrow F\rightarrow G_{\text{\O}}^T(K)^\mathcal{C}\rightarrow 1$. This sequence is obtained from the sequence $1\rightarrow \ker\phi\rightarrow \mathcal{F}_n^{\mathcal{C}}\rightarrow G_{\text{\O}}^T(K)^\mathcal{C}\rightarrow 1$ by quotienting out the first two terms by the intersection of maximal $\mathcal{F}_n^{\mathcal{C}}\rtimes\Gamma$-normal subgroups of $\ker\phi$. We get a presentation as in the theorem precisely when the multiplicity $m$ of every irreducible $G_{\text{\O}}^T(K)^\mathcal{C}\rtimes\Gamma$-module $A$ in $R$ is not ``too large." This is quantified by the inequality 
    \[\left(1-\frac{h_{F\rtimes\Gamma}(A,A)^m|A^\Gamma|}{|Y(A)|^n\prod_{i=1}^{|T|}|A^{\Gamma_{v_i}}|}\right)\geq 0,\]
    which is obtained from \Cref{prop: relation probability}. This simplifies to 
    \[n\dim_{\mathbb{F}_\ell}A/A^\Gamma+\sum_{i=1}^{|T|}\dim_{\mathbb{F}_\ell}A^{\Gamma_{v_i}}\geq m\Hom_{G(K^\#/K)^\mathcal{C}\rtimes\Gamma}(A,A)+\dim_{\mathbb{F}_\ell}A^\Gamma,\]
    which is equivalent to 
    \[m\leq \frac{n\dim_{\mathbb{F}_\ell}A-(n+1)\dim_{\mathbb{F}_\ell}A^\Gamma+\sum_{v\in T} \dim_{\mathbb{F}_\ell}A^{\Gamma_v}}{\dim_{\mathbb{F}_\ell}\Hom_{G(K^\#/K)^\mathcal{C}\rtimes\Gamma}(A,A)}.\]

    We claim that $A\not\cong_{G(Q^{\text{sep}}/Q)}\mu_\ell$. Indeed, we made sure that groups in $\mathcal{C}$ have order coprime to $|\mu(K)|$, so the only way for $A\cong_{G(Q^{\text{sep}}/Q)}\mu_\ell$ is for $K(\mu_\ell)/K$ to be a nontrivial subextension of $(K_{\text{\O}}^T)^{\mathcal{C}}/K$, the extension associated to the Galois group $G_{\text{\O}}^T(K)^{\mathcal{C}}$. By the admissibility of $G_{\text{\O}}^T(K)^{\mathcal{C}}$, $G(K(\mu_\ell)/K)$ is an admissible $\Gamma$-group, with the $\Gamma$-action induced by conjugation. But $K(\mu_\ell)$ is the composite of Galois extensions $Q(\mu_\ell)/Q$, $K/Q$, the former being abelian, so the $\Gamma$-action on $G(K(\mu_\ell)/K)$ is trivial. This contradicts admissibility, which gives that $G(K(\mu_\ell)/K)_{\Gamma}=1$. 

    Since $A$ is a subquotient of $\mathcal{F}_n^{\mathcal{C}}$, the module $A$ is odd and coprime to $|\Gamma|$. Thus, the theorem follows by \Cref{cor: multiplicity bound}.
 \end{proof}

 From the perspective of presentations of profinite groups, the presentation above may not look pleasant. One way to mitigate this is to replace $\mathcal{F}_n$ with $F_n$. In particular, by Proposition $3.4$ and Corollary $5.3$ of \cite{Liu2} applied to the surjection $F_n\twoheadrightarrow G$ obtained in \Cref{cor: multiplicity bound}, the multiplicity of an irreducible odd $|\Gamma|'$-$G_{\text{\O}}^T(K)^\mathcal{C}\rtimes\Gamma$ module $A$ in the presentation $F_n^{\mathcal{C}}\twoheadrightarrow G_{\text{\O}}^T(K)^\mathcal{C}$ is bounded from above by 
 \[\frac{n\dim_{\mathbb{F}_\ell}A-\dim_{\mathbb{F}_\ell}A^\Gamma+\sum_{v\in T}A^{\Gamma_v}}{\dim_{\mathbb{F}_\ell}\Hom_{G(K^\#/K)^\mathcal{C}\rtimes\Gamma}(A,A)}.\]
 This suggests that we consider the presentation $F_n/[Y(r_1), r_2,\dots, r_{n+|T|}]$, where $r_i\in F_n^{\Gamma_{v_{n-i}}}$ for $n+1\leq i\leq n+|T|$. Indeed it is possible to show that $G_{\text{\O}}^T(K)^\mathcal{C}$ has a presentation of the form $F_n^\mathcal{C}/[Y(r_1), r_2,\dots, r_{n+|T|}]$. However, the distribution of random groups (discussed in the next section) of the form $F_n/[Y(r_1), r_2,\dots, r_{n+|T|}]$ will not be the correct distribution that predicts the distribution of the pro-prime-to-$2|\Cl_T(Q)||\Gamma|$ completion of $G_{\text{\O}}^T(K)$ because these random groups do not have to be admissible. In particular, some $H$-moments for nonadmissible groups $H$ would become nonzero, which we need to avoid. 
 \begin{remark}\label{rmk: class group presentation}
    In the set up of \Cref{thm: main}, take $Q$ to be a number field and $\mathcal{C}$ to be the set of all finite abelian $\Delta'$-$\Gamma$-groups. Then 
    \begin{equation}
        G_{\text{\O}}^T(K)^{\mathcal{C}}\cong \Cl_T(K)_{\Delta'},
    \end{equation}
    where we recall that the notation $G_{\Delta'}$ denotes the pro-prime-to-$\Delta$ completion of $G$ for any profinite group $G$. Therefore, \Cref{thm: main} tells us that 
    \begin{equation}\label{eq: T-class group presentation}
        \Cl_T(K)_{\Delta'}\cong \mathcal{F}_n(\Gamma)^{\text{ab}}_{\Delta'}/[r_i^{-1}\gamma(r_i),r_j].
    \end{equation}
    Set $Z:=\prod_{p\nmid \Delta}\mathbb{Z}_p$ and $B:=Z[\Gamma]/(\sum_{\gamma\in\Gamma}\gamma)$. We note that $B$ depends on $T,K,$ and $Q$. Corollary $3.10$ of \cite{lwzb} states that the right hand side of \eqref{eq: T-class group presentation} is $B^n/[\gamma(r_i)-r_i, r_j]$ written additively. Using that $|\Gamma|$ is invertible in $B$ and that the norm element acts by multiplication by $0$, we obtain 
    \begin{equation}\label{eq: T-class group presentation simplified}
        \Cl_T(K)_{\Delta'}\cong B^n/[r_1,r_2,\dots, r_{n+|T|}]
    \end{equation}
    where the relations $r_i\in B^n$ for $1\leq i\leq n$, and $r_j\in (B^n)^{\Gamma_{v_{j-n}}}$ for $n+1\leq j\leq n+|T|$. Recalling that $\Delta_\infty:=|\Cl(Q)||\mu(K)||\Gamma|$ and taking $T$ to be the set of all archimedean places of $Q$ gives 
    \begin{equation}
        \Cl(K)_{\Delta_\infty'}\cong B^n/[r_1,r_2,\dots, r_{n+r+s}].
    \end{equation}
    
    Writing the class group in this form is important because in the case when $Q=\mathbb{Q}$ and $\Gamma$ has order $2$, we recover the classical random matrix interpretation of the Cohen-Lenstra heuristics as seen in Cohen and Lenstra's work \cite{cl}. This is also explained in Ellenberg and Venkatesh's article in Section $4.1$ of \cite{EV} and Wood's argument before Theorem $3.6$ in \cite{Woo3}, as long as we make sure to avoid primes that can come up as roots of unity of $K$. 
 \end{remark}
 The next proposition shows that \eqref{eq: T-class group presentation} and \eqref{eq: T-class group presentation simplified} in \Cref{rmk: class group presentation} can be obtained from more classical results without using \Cref{thm: main}. We repeat the isomorphisms here for convenience. 
 \begin{proposition}\label{prop: class group presentation alt}
    Suppose $K/Q$ is a $\Gamma$-extension of number fields. Then
    \begin{equation}
        \Cl_T(K)_{\Delta'}\cong B^n/[r_1,r_2,\dots, r_{n+|T|}]
    \end{equation}
    where $r_i\in (B^n)^{\Gamma_{v_{j-n}}}$ for $n+1\leq j\leq n+|T|$, and
    \begin{equation}
        \Cl_T(K)_{\Delta'}\cong \mathcal{F}_n(\Gamma)^{\text{ab}}_{\Delta'}/[r_i^{-1}\gamma(r_i),r_j],
    \end{equation}
    where $1\leq i \leq n+1$, $n+2\leq j\leq n+|T|$,  $r_{n+1}\in (\mathcal{F}_n(\Gamma)^{\text{ab}}_{\Delta'})^{\Gamma_{v_1}}$, and $r_j\in (\mathcal{F}_n(\Gamma)^{\text{ab}}_{\Delta'})^{\Gamma_{v_{j-n}}}$ for $n+2\leq j\leq n+|T|$. 
 \end{proposition}
 \begin{proof}[Proof of \Cref{prop: class group presentation alt} without using \Cref{thm: main}]
 The arguments of this proof can be seen in the paragraphs before Theorem 3.6 of \cite{Woo3}. 

 The reciprocity map and the Cebotarev density theorem ensure that for every ideal class of $\Cl(K)$, there is a prime ideal in that ideal class that is totally split in $K/Q$. Thus, let $S_f$ be a finite $\Gamma$-invariant set of finite places of $K$ that is totally split in $K/Q$, and $S_f$ generates $\Cl(K)$ as a $\mathbb{Z}$-module. Let $S:=S_f\cup S_{\infty}$, where $S_\infty$ is the set of archimedean primes of $K$. Finally, let $I_S$ be the free $Z[\Gamma]$-module generated by one element from each $\Gamma$-orbit $S_f$, and $\mathcal{O}_{K,S}^\times$ be the group of $S$-units (i.e. elements $a$ so that for every prime $\mathfrak{p}\not\in S$, $v_\mathfrak{p}(a)=0$). 

 By Corollary 8.7.3 of \cite{NSW}, we have a surjective $Z[\Gamma]$-equivariant map 
 \[\bigoplus_{\mathfrak{p}\in S}Z\twoheadrightarrow \prod_{p\nmid \Delta}\mathcal{O}_{K,S}^\times\otimes \mathbb{Z}_p\]
 because $Z$ is relatively prime to the order of $\mu(K)$ and $\Gamma$. Here the $\Gamma$-action on $\bigoplus_{\mathfrak{p}\in S}Z$ is induced from its action on $S$. 

 The cokernel of the canonical $\Gamma$-equivariant map $ \prod_{p\nmid \Delta}\mathcal{O}_{K,S}^\times\otimes\mathbb{Z}_p\rightarrow I_S$ induced by $a\mapsto (a)$ is the prime-to-$\Delta$ completion of the ideal class group. To obtain $\Cl_T(K)$ instead, we can further quotient by the image of $\bigoplus_{\mathfrak{p}\in T(K)\backslash S_{\infty}}Z$ under the map sending a generator of the component corresponding to $\mathfrak{p}\in T(K)\backslash S_{\infty}$ to an element in $I_S$ whose image in $\Cl(K)$ is the class generated by $\mathfrak{p}$. In other words, we have 
 \[\Cl_T(K)_{\Delta'}\cong \coker(\bigoplus_{\mathfrak{p}\in S\sqcup (T(K)\backslash S_\infty)}Z\rightarrow I_S).\]
 Moreover, the norm acts trivially on the image of every $Z[\Gamma]$-component of $I_S$ in this quotient of $\Cl_T(K)$ because we assumed that $|\Cl_T(Q)|$ divides $\Delta$. Thus, 
 \[\Cl_T(K)_{\Delta'}\cong (Z[\Gamma]/(\sum_{\gamma\in\Gamma}\gamma))^{\oplus|S_f(Q)|}/ (r_1,r_2,\dots, r_{|S_f(Q)|+|T|})\]
 where we take $S_f(Q)$ to be the set of places of $Q$ below $S_f$, and the elements $r_i$ to be images of a generator of a component $Z$ in $\bigoplus_{\mathfrak{p}\in S\sqcup (T(K)\backslash S_\infty)}Z$. We pick one component corresponding to each $\Gamma$-orbit of $S\sqcup (T(K)\backslash S_\infty)=S_f\sqcup T(K)$. 

 For $r_i$ corresponding to a place $\mathfrak{q}$ in $T$, take a place $\mathfrak{p}$ in $T(K)$ above it, with decomposition group $\Gamma_{\mathfrak{p}|\mathfrak{q}}$. This group acts trivially on the component of $\bigoplus_{\mathfrak{p}\in S_f\sqcup T(K)}Z$ corresponding to $\mathfrak{p}$, so $\Gamma_{\mathfrak{p}|\mathfrak{q}}$ acts trivially on $r_i$ as well. Picking places $\mathfrak{p}$ corresponding to the decomposition groups $\Gamma_{v_j}$ where $v_j=\mathfrak{q}$ yields \eqref{eq: T-class group presentation simplified} after using $n=|S_f(Q)|$.

 As in the argument in \Cref{rmk: class group presentation}, we can replace every instance of $r_i$ with $\gamma(r_i)-r_i$ for all $\gamma\in\Gamma$. This finally gives us the presentation 
 \[\Cl_T(K)_{\Delta'}\cong B^{n}/(\gamma(r_i)-r_i, r_j)\]
 where $1\leq i \leq n+1$, $n+2\leq j\leq n+|T|$, giving us \Cref{eq: T-class group presentation} after using Corollary $3.10$ of \cite{lwzb} one more time.  
 \end{proof}
 \begin{corollary}
    Suppose $K/Q$ is a $\Gamma$-extension of number fields.  Then 
    \begin{equation}
        \Cl(K)_{\Delta_\infty'}\cong B^n/[r_1,\dots, r_{n+r+s}]
    \end{equation}
    where $r_j\in (B^n)^{\Gamma_{v_{j-n}}}$ for $n+1\leq j\leq n+r+s$, and 
    \begin{equation}
        \Cl(K)_{\Delta_\infty'}\cong \mathcal{F}_n(\Gamma)^{\text{ab}}_{\Delta_\infty'}/[r_i^{-1}\gamma(r_i),r_j].
    \end{equation}
    where $1\leq i \leq n+1$, $n+2\leq j\leq n+r+s$,  $r_{n+1}\in (\mathcal{F}_n(\Gamma)^{\text{ab}}_{\Delta_\infty'})^{\Gamma_{v_1}}$, and $r_j\in (\mathcal{F}_n(\Gamma)^{\text{ab}}_{\Delta_\infty'})^{\Gamma_{v_{j-n}}}$ for $n+2\leq j\leq n+r+s$. 
 \end{corollary}

\section{\texorpdfstring{Conjectures on $H$-moments of $G_{\text{\O}}^T(K)^{\mathcal{C}}$}{Conjectures on H-moments of GOT(K)C}}\label{sec: conjectures}
 In this section, we state the conjectured distribution of a certain natural quotient of $G_{\text{\O}}^T(K)$ as $K/Q$ varies over $\Gamma$-extensions of $Q$ with prescribed local conditions. We will also discuss the subtlety coming from intermediate cyclotomic subfields of $K/Q$, and why it is reasonable to ignore these intermediate extensions when $\Gamma$ is abelian. 

 Let $T$ be as before. For $D$ a positive integer, let $E_{\Gamma,(\Gamma_v)_{v\in T}}(D,Q)$ be the set of isomorphism classes of $\Gamma$-extensions of $Q$ with decomposition group of primes $v\in T$ conjugate to $\Gamma_v$, and the absolute norm of the product of ramified primes equal to $D$ (denoted $\rDisc K/Q$). Set 
 \begin{equation}E_{\Gamma, (\Gamma_v)_{v\in T}}(Q):=\bigcup_{D\geq 0}E_{\Gamma, (\Gamma_v)_{v\in T}}(D,Q).\end{equation} 
 In other words, $E_{\Gamma, (\Gamma_v)_{v\in T}}(Q)$ consists of $\Gamma$-extensions with decomposition group at places $v\in T$ conjugate to $\Gamma_v$. This set does not depend on how we order the extensions. 

 We also set 
 \begin{equation}
    \Delta_Q:=\begin{cases}|\Cl_T(Q)||\mu(Q)||\Gamma| &\text{if }Q\text{ is a number field}\\
    \text{char} (Q)|\Cl_T(Q)||\mu(Q)||\Gamma|&\text{if }Q\text{ is a function field}
    \end{cases}
 \end{equation} where we again make the dependence on $\Gamma$ and $T$ implicit. Note that $\Delta_Q\mid\Delta$. 
 \begin{definition}\label{def: E'}
    Let $E'_{\Gamma,(\Gamma_v)_{v\in T}}(D,Q)\subset E_{\Gamma,(\Gamma_v)_{v\in T}}(D,Q)$ be the subset of extensions satisfying the following additional properties:
 \begin{enumerate}
    \item $K/Q$ is linearly disjoint from $Q(\mu_{\exp|\Gamma|})$, where $\exp|\Gamma|$ is the least common multiple of the orders of elements of $\Gamma$. 
    \item For any prime $\ell$ coprime to $\Delta_Q$, we have that $\ell$ and $|\mu(K)|$ are coprime as well. 
 \end{enumerate}
 \end{definition}

 This corresponds to removing a thin set, as in Definition $2.5$ of \cite{LS}. Note that there are only finitely many primes $\ell$ that do not satisfy property $(2)$ above as we vary among extensions in $E_{\Gamma, (\Gamma_v)_{v\in T}}(D,Q)$, because any such prime satisfies the property that $\ell-1\mid |\Gamma|[Q:\mathbb{Q}]$ if $Q$ is a number field, and the constant field of $K$ is contained in $\mathbb{F}_{q^{|\Gamma|}}$ when $Q$ is a function field and $\mathbb{F}_q$ is the constant field of $Q$. 

 Then we propose the following conjecture:
 \begin{conjecture}\label{conj: main}
    Let $H$ be a finite $\Delta_Q'$-$\Gamma$-group. If $\#E'_{\Gamma,(\Gamma_v)_{v\in T}}(Q)\neq0$ and $H$ is admissible (otherwise the $H$-moment is $0$), then 
    \begin{equation}\lim_{X\rightarrow\infty}\frac{\sum_{D\leq X}\sum_{K\in E'_{\Gamma,(\Gamma_v)_{v\in T}}(D,Q)}|\Sur_\Gamma(G_{\text{\O}}^T(K)_{|\Gamma|'},H)|}{\sum_{D\leq X}\#E'_{\Gamma,(\Gamma_v)_{v\in T}}(D,Q)}=\frac{|H^\Gamma|}{\prod_{v\in T}|H^{\Gamma_v}|}.\end{equation}
 \end{conjecture}
 \begin{remark}
    If we take $Q$ to be a number field, $T$ to be all archimedean places of $Q$, and replace $E'_{\Gamma, (\Gamma_v)_{v\in T}}(D,Q)$ with $E_{\Gamma, (\Gamma_v)_{v\in T}}(D,Q)$ then we obtain a relative version of the moment conjectures given in Conjecture 1.3 of \cite{lwzb} in the real case and Conjecture 1.3 in \cite{LW2} in the imaginary case where $Q=\mathbb{Q}$, and indeed this specialization reduces to these two cases. These conjectures are based on a random group construction in \Cref{sec: random} mimicking that of \cite{lwzb} and the random group model in the imaginary case in \cite{LW2}. This also gives a special case of Conjecture $1.3$ of \cite{SW2}, which goes further by considering the effects of roots of unity $\mu(Q)$, as remarked in the Introduction.
 \end{remark}
 \begin{remark}\label{conj: Fq(t) case}
    If we take $Q=\mathbb{F}_q(t)$, $T=\{\infty\}$, and replace $E'_{\Gamma, (\Gamma_v)_{v\in T}}(D,Q)$ with $E_{\Gamma, (\Gamma_v)_{v\in T}}(D,Q)$, then this recovers the conjectured non-abelian Cohen-Lenstra heuristics for function fields as stated in Conjecture $7.1$ of \cite{LW2}. In fact, this generalizes the conjecture in \cite{LW2} by allowing arbitrary local conditions at $\infty$ as opposed to the imaginary case where we required the inertia group to equal the decomposition group. In this case, the conjectured moments are once again $\frac{1}{[H^{\Gamma_{\infty}}:H^\Gamma]}$.
 \end{remark}

 In an ideal world, we would conjecture a version of \Cref{conj: main} where all instances of $E'_{\Gamma, (\Gamma_v)_{v\in T}}(D,Q)$ are replaced by $E_{\Gamma, (\Gamma_v)_{v\in T}}(D,Q)$. However, such a conjecture would require the heuristic that the probability that a $\Gamma$-extension in $E_{\Gamma, (\Gamma_v)_{v\in T}}(D,Q)$ picks up extra roots of unity other than $\mu(Q)$ or is not linearly disjoint from $Q(\mu_{\exp|\Gamma|})$, is $0$. This is known for abelian $\Gamma$-extensions of number fields. The following proposition is well-known to the arithmetic statistics community, but I could not find a reference for it. I thank Peter Koymans for informing me that this result was not in the literature. 

 \begin{proposition}\label{prop: intermediate fields}
    Let $Q$ be a number field, $\Gamma$ be a finite abelian group, $T$ be a finite set of primes of $Q$, $\Gamma_v\subset \Gamma$ a subgroup for each $v\in T$, and $E/Q$ a nontrivial finite abelian extension. Then filtered by $\rDisc K/Q$ (or any fair counting function $C$, see Section 2.1 of \cite{Woo}), the probability that a $\Gamma$-extension $K/Q$ with decomposition group $\Gamma_v$ for each $v\in T$ contains $E$ is $0$. That is, when $\#E_{\Gamma, (\Gamma_v)_{v\in T}}(Q)\neq 0$,
    \begin{equation}
        \lim_{X\rightarrow \infty}\frac{\#\{K\in E_{\Gamma,(\Gamma_v)_{v\in T}}(Q)\text{ so that } C(K)\leq X,\: K\supset E\}}{\#\{K\in E_{\Gamma,(\Gamma_v)_{v\in T}}(Q)\text{ so that } C(K)\leq X\}}=0.
    \end{equation}
 \end{proposition}
 \begin{proof}
    We will freely use terminology from Section $2.1$-$2.3$ of \cite{Woo}. Assume $E_{\Gamma, (\Gamma_v)_{v\in T}}(Q)\neq \text{\O}$. Recall that there is a subset $S_0$ of primes above $2$ in $Q$ so that a local specification $\Sigma$ on a finite set of primes $S$ of a $\Gamma$-extension $K/Q$ is viable if and only if $S_0\not\subset S$ or $\Sigma$ restricts to a local specification $\Sigma(i)$ on $S_0$ for some finite nonempty list of local specifications $\Sigma(1),\dots, \Sigma(\ell)$ on $S_0$.

    Set $T':=T\cup S_0$. Then there is a finite list of all local specifications $\Sigma'(1), \Sigma'(2),\dots, \Sigma'(n)$ on $T'$ that is viable and at the primes $v\in T$ the $\Gamma$-structured $Q_v$-algebra is isomorphic to $\Ind_\Gamma^{\Gamma_v}M$, where $M/Q_v$ is a $\Gamma_v$-extension of $v$-adic fields. This list is nonempty from our assumption $E_{\Gamma, (\Gamma_v)_{v\in T}}(Q)\neq \text{\O}$. Any $\Sigma'(j)$ restricts to a viable local specification $\Sigma(i)$ on $S_0$. 

    Next, set $S$ to be a finite set of primes disjoint from $T'$ and primes dividing $|\Gamma|$. Then any local specification on $T'\cup S$ is viable if and only if it restricts to a viable local specification $\Sigma'(j)$ on $T'$. Let $\Omega_v$ be the set of all local specifications on a place $v$, and $A_v\subset \Omega_v$ be the subset of local specifications $\Sigma$ so that $M\supset E_v$, where $M$ and $\Sigma$ are of the form in the previous paragraph for some subgroup $\Gamma_v$ that we do not fix. Here $E_v$ is the completion of $E$ at a place above $v$. For a local specification $\Sigma$ or a set of local specifications $B$ at a fixed set of places, denote by $\widetilde{\Sigma}$ to be the set of tuples of local specifications at all primes restricting to $\Sigma$, and $\widetilde{B}$ to be the set of tuples of local specifications at all primes restricting to a local specification in $B$.
        
    In the notation of \cite{Woo}, we have the following inequality from the fact that $K\supset E$ implies $M=K_v\supset E_v$:

    \begin{equation}
        \limsup_{X\rightarrow \infty}\frac{\#\{K\in E_{\Gamma,(\Gamma_v)_{v\in T}}(Q)\text{ so that } C(K)\leq X,\: K\supset E\}}{\#\{K\in E_{\Gamma,(\Gamma_v)_{v\in T}}(Q)\text{ so that } C(K)\leq X\}}\leq \frac{\sum_{j=1}^n\text{Pr}_C(\widetilde{\Sigma'(j)}\cap \bigcap_{v\in S}\widetilde{A_v})}{\sum_{j=1}^n\text{Pr}_C(\widetilde{\Sigma'(j)})}.
    \end{equation}
    By Theorem $2.1$ of \cite{Woo}, $\Sigma'(j)$ specializing to $\Sigma(i)$ for some $i$, and independence of the probability $P_C$ for pairwise distinct primes, we obtain 
        \[\text{Pr}_C(\widetilde{\Sigma'(j)}\cap \bigcap_{v\in S}\widetilde{A_v})=\frac{P_C(\widetilde{\Sigma'(j)}\cap \bigcap_{v\in S}\widetilde{A_v})}{P_C(\bigcup_{i=1}^\ell\widetilde{\Sigma(i)})}=\frac{P_C(\widetilde{\Sigma'(j)})\prod_{v\in S}P_C(\widetilde{A_v})}{P_C(\bigcup_{i=1}^\ell\widetilde{\Sigma(i)})}.\]
    Similarly, we obtain 
        \[\text{Pr}_C(\widetilde{\Sigma'(j)})=\frac{P_C(\widetilde{\Sigma'(j)})}{P_C(\bigcup_{i=1}^\ell\widetilde{\Sigma(i)})}>0.\]

    The Cebotarev density theorem guarantees that there are infinitely many primes $v$ so that $E_v\supsetneq Q_v$. Thus, if we further impose that $v\in S$ satisfies $E_v\supsetneq Q_v$, we have $\Ind_G^1 Q_v\not\in A_v$. Using the definition of the probability $P_C$, we have 
        \[P_C(\widetilde{A_v})\leq P_C(\widetilde{\Omega_v}-\widetilde{\Ind_G^1Q_v})=1-P_C(\widetilde{\Ind_G^1Q_v})=1-\frac{1}{\sum_{L\in \Omega_v}(Nv)^{-c(L)/m}}\]
    where $Nv$ is the absolute norm of the finite prime $v$ of $Q$, $c$ is the nonnegative function giving rise to the counting function $C$ in \cite{Woo}, and $m$ is some positive integer depending on the group $\Gamma$. Utilizing local class field theory and the correspondence between $\Gamma$-structured \'{e}tale $Q_v$-algebras and continuous homomorphisms $G_{Q_v}\rightarrow \Gamma$ established in Lemma $2.6$ of \cite{Woo},
        \[1-\frac{1}{\sum_{L\in \Omega_v}(Nv)^{-c(L)/m}}\leq 1-\frac{1}{\#\Hom(Q_v^{\times},\Gamma)}=1-\frac{1}{|\Gamma||\Gamma[|\mu'(Q_v)|]|}\leq 1-\frac{1}{|\Gamma|^2},\]
    where $\mu'(Q_v)$ is the set of roots of unity in $Q_v$ whose order is prime to the residue characteristic, and $\Gamma[|\mu'(Q_v)|]$ is the set of elements of $\Gamma$ whose order divides $|\mu'(Q_v)|$. This upper bound is independent of the place $v\in S$. Therefore, 

    \begin{equation}
        \text{Pr}_C(\widetilde{\Sigma'(j)}\cap \bigcap_{v\in S}\widetilde{A_v})\leq \frac{P_C(\widetilde{\Sigma'(j)})}{P_C(\bigcup_{i=1}^\ell\widetilde{\Sigma(i)})}\left(1-\frac{1}{|\Gamma|^2}\right)^{|S|}.
    \end{equation}
    Taking $|S|\rightarrow\infty$ gives $\Pr_C(\widetilde{\Sigma'(j)}\cap \bigcap_{v\in S}\widetilde{A_v})\rightarrow0$, and hence 
        \[\limsup_{X\rightarrow \infty}\frac{\#\{K\in E_{\Gamma,(\Gamma_v)_{v\in T}}(Q)\text{ so that } C(K)\leq X,\: K\supset E\}}{\#\{K\in E_{\Gamma,(\Gamma_v)_{v\in T}}(Q)\text{ so that } C(K)\leq X\}}=0,\]
    which finishes the proof. 
    \end{proof}
 \begin{remark}
    The proof above shows that for a fixed number field $Q$, a finite abelian group $\Gamma$, and a viable local specification $\Sigma$ at finitely many primes, the probability that an abelian $\Gamma$-extension $K/Q$ with local specification $\Sigma$ contains a fixed nontrivial extension $E/Q$ is $0$ as well, filtered by any fair counting function $C$. 
 \end{remark}
 \begin{corollary}
    Let $Q$ be a number field, $\Gamma$ be a finite abelian group, $T$ be a finite set of primes of $Q$, and $\Gamma_v\subset \Gamma$ a subgroup for each $v\in T$. Then filtered by $\rDisc K/Q$ (or any fair counting function $C$), the probability that a $\Gamma$-extension $K/Q$ with decomposition group $\Gamma_v$ for each $v\in T$ picks up extra roots of unity or is not linearly disjoint from $Q(\mu_{\exp|\Gamma|})$, is $0$. 
 \end{corollary}
    \begin{proof}
    Among the abelian extensions $Q(\mu_n)/Q$, only finitely many of them have Galois group that is a quotient of the finite abelian group $\Gamma$. By \Cref{prop: intermediate fields}, the probability that a $\Gamma$-extension in $E_{\Gamma, (\Gamma_v)_{v\in T}}(Q)$ contains one of these finitely many possible nontrivial extensions $Q(\mu_n)/Q$ or a nontrivial subextension of $Q(\mu_{\exp|\Gamma|})$ is $0$ by finite subadditivity.  
    \end{proof}
 This corollary shows that for abelian $\Gamma$, the probability of containing a bad intermediate extension that we looked at in \Cref{def: E'} is $0$, so it is reasonable to replace every instance of $E'_{\Gamma, (\Gamma_v)_{v\in T}}(D,Q)$ with $E_{\Gamma, (\Gamma_v)_{v\in T}}(D,Q)$ in \Cref{conj: main}. 
 \begin{remark}\label{rmk: malle}
    The denominator in \Cref{conj: main} is closely related to Malle's conjecture \cite{Mal}, \cite{Mal2}. Like the Cohen-Lenstra-Martinet heuristics, Malle's conjecture has been a guiding force in arithmetic statistics research in the last two decades. 

    Malle's conjecture for abelian $\Gamma$---when ordering extensions by any fair counting function (even with viable local specifications)---is proven in \cite{Woo}, but it is known to have counterexamples for nonabelian $\Gamma$ related to the presence of cyclotomic subextensions of $Q(\mu_{\exp|G|})$. Kl\"uners was the first to provide such a counterexample when counting $C_3\wr  C_2$-extensions by discriminant in \cite{Klu}, and Koymans, Pagano, and Wang provide counterexamples in \cite{KP}, \cite{Wan} even when ordered by $\rDisc K/Q$. 

    The problem with Malle's conjecture in these counterexamples is the ``Big Fiber" problem: a positive proportion of $\Gamma$-extensions contain a nontrivial cyclotomic subextension of $Q(\mu_{\exp |\Gamma|})$. These extensions can contribute a higher power of $\log$ than the conjectured power in Malle's conjecture. Given such an extension $K/Q$ and $E=Q(\mu_{\exp|\Gamma|})\cap K\supsetneq Q$, thinking of $G_{\text{\O}}(K)_{|\Gamma|'}$ as a random $\Gamma$-group is unreasonable, since we have not taken into account the extra deterministic information that $K\supsetneq E$. Therefore, the naive fix is to get rid of these extensions from $E_{\Gamma,(\Gamma_v)_{v\in T}}(D,Q)$, and that is exactly what we do in \Cref{conj: main}. We note that getting rid of this thin subset of extensions can remove $100\%$ of $\Gamma$-extensions, and inevitably, we would be trying to understand the statistics of unramified extensions of $0\%$ of $\Gamma$-extensions. 
 \end{remark}
    In the formulation of Malle's conjectures in \cite{LS} and \cite{Wan}, roots of unity coming from property $(2)$ are not taken into account. Therefore, if we define $E''_{\Gamma, (\Gamma_v)_{v\in T}}(D,Q)$ to be the subset of $E_{\Gamma, (\Gamma_v)_{v\in T}}(D,Q)$ satisfying property $(1)$ in the definition of $E'_{\Gamma, (\Gamma_v)_{v\in T}}(D,Q)$, it may be reasonable to conjecture the following:

 \begin{conjecture}\label{conj: E''}
    \Cref{conj: main} is true with every instance of $E'_{\Gamma, (\Gamma_v)_{v\in T}}(D,Q)$ replaced by $E''_{\Gamma, (\Gamma_v)_{v\in T}}(D,Q)$. 
 \end{conjecture}
 \begin{remark}\label{rmk: intermediate fields}
    Of course, if any intermediate field $E$ appears with positive probability in the set of $\Gamma$-extensions even after removing the bad cyclotomic subextensions as in \Cref{rmk: malle}, then we should also get rid of these extensions in principle. However, the author does not know of any such examples where $E\not\subset Q(\mu_{\exp|\Gamma|})$, so we did not consider such extensions in $E'_{\Gamma, (\Gamma_v)_{v\in T}}(D,Q)$. However, if these extensions exist, \Cref{conj: main} will be unlikely to be correct. One possible remedy is to look for a counting function $C$ on the set of $\Gamma$-extensions $K/Q$ so that the probability that $K/Q$ contains any nontrivial intermediate field is $0$, and give \Cref{conj: main} filtered by $C$ instead of $\rDisc K/Q$. In this case, we can replace $E'_{\Gamma, (\Gamma_v)_{v\in T}}(D,Q)$ with $E_{\Gamma, (\Gamma_v)_{v\in T}}(D,Q)$ if we also replace $\rDisc K/Q=D$ with $C(K)=D$. For completeness, we state this conjecture as well. 
 \end{remark}
    Denote by $E_{\Gamma, (\Gamma_v)_{v\in T}}^C(D,Q)$ the subset of $E_{\Gamma, (\Gamma_v)_{v\in T}}(Q)$ consisting of extensions $K/Q$ where $C(K)=D$ for some counting function $C$.   
 \begin{conjecture}\label{conj: main 2}
    Let $H$ be a finite $\Delta_Q'$-$\Gamma$-group. If $\#E_{\Gamma, (\Gamma_v)_{v\in T}}(Q)\neq 0$ and $H$ is admissible, then for some counting function $C$ on $E_{\Gamma, (\Gamma_v)_{v\in T}}(Q)$ satisfying the property discussed in \Cref{rmk: intermediate fields} not depending on $H$, we have
    \begin{equation}\lim_{X\rightarrow\infty}\frac{\sum_{D\leq X}\sum_{K\in E_{\Gamma,(\Gamma_v)_{v\in T}}^C(D,Q)}|\Sur_\Gamma(G_{\text{\O}}^T(K)_{|\Gamma|'},H)|}{\sum_{D\leq X}\#E_{\Gamma,(\Gamma_v)_{v\in T}}^C(D,Q)}=\frac{|H^\Gamma|}{\prod_{v\in T}|H^{\Gamma_v}|}.\end{equation}
 \end{conjecture}
 When $T=S_{\infty}$, this conjecture is a special case of Conjecture $1.3$ in \cite{SW2}. Sawin and Wood allow $|H|$ to be coprime to just $|\Gamma|$ instead of $\Delta_Q$, and take into account the effects of roots of unity in the base field in the moment formulas.   
            
 We look forward to future research in understanding the relationship between presence of intermediate fields in $\Gamma$-extensions, roots of unity, and different counting functions.  
\section{Random group model}\label{sec: random}
 In this section, we define the random group model that is the basis of \Cref{conj: main}. We construct this by taking a limit of Borel probability measures $\mu_{n, \underline{\Gamma}}$, and showing that the limit $\mu_{\underline{\Gamma}}$ is indeed a countably additive Borel probability measure. This section is analogous to Sections $4$ and $5$ of \cite{lwzb} and outlined in Section $6$ of \cite{LW2}. 
 \begin{definition}\label{def: random group}
    Let $\Gamma$ be a finite group and let $\Gamma_1,\Gamma_2, \dots, \Gamma_{u+1}$ be a fixed ordered tuple of subgroups of $\Gamma$, for $u\geq 0$. We will denote this tuple by $\underline{\Gamma}:=(\Gamma_i)_{1\leq i\leq u+1}$. For a positive integer $n$, define a random $(\Gamma, \underline{\Gamma})$-group $X_{n,\underline{\Gamma}}:=\mathcal{F}_n/[Y(S_1),S_2]$ where $S_1$ is a tuple of $n+1$ elements $(x_1,x_2,\dots, x_n, x_{n+1})$ where $x_i$ for $1\leq i \leq n$ is chosen independently and uniformly at random with respect to Haar measure on $\mathcal{F}_n$, $x_{n+1}$ is chosen randomly with respect to the Haar measure on $\mathcal{F}_n^{\Gamma_1}$, and $S_2$ is a tuple of $u$ elements $(x_{n+2},\dots, x_{n+u+1})$ chosen uniformly at random with respect to the Haar measure on $\mathcal{F}_n^{\Gamma_2}\times \mathcal{F}_n^{\Gamma_3}\times \cdots\times \mathcal{F}_n^{\Gamma_{u+1}}$. 
 \end{definition}
 In the above definition, the dependence on $\Gamma$ is implicit. Directly from the formulas for the probabilities, it will turn out that the distribution does not depend on the ordering of the groups $\Gamma_i$. 

 We set $\mathcal{P}$ to be the set of isomorphism classes of admissible $\Gamma$-groups $X$ which satisfy the property that for any finite set $\mathcal{C}$ of $|\Gamma|'$-$\Gamma$-groups, $X^{\mathcal{C}}$ is finite, as in Section 3.2 of \cite{lwzb}. Moreover, $X_{n,\underline{\Gamma}}\in \mathcal{P}$ since it is a $\Gamma$-quotient of $\mathcal{F}_n$, which also lies in $\mathcal{P}$. By the comment after \Cref{cor: multiplicity bound}, $G_{\text{\O}}^T(K)_{\Delta_Q'}$ lies in $\mathcal{P}$ as well. The space $\mathcal{P}$ is equipped with a topology defined by basic open sets $U_{\mathcal{C},H}:=\{X\in \mathcal{P}: X^{\mathcal{C}}\cong H\}$ for all $\mathcal{C}$ finite and $H$ a finite $\Gamma$-group. 

 \begin{definition}\label{def: measure mun}
    For a Borel set $A\subset \mathcal{P}$, define the measure 
    \begin{equation}\mu_{n,\underline{\Gamma}}(A):=\Prob\left(X_{n,\underline{\Gamma}}\in A\right).\end{equation}
    Thus, $\mu_{n,\underline{\Gamma}}(A)$ is the probability that when we pick $(S_1,S_2)=((x_1,\dots,x_{n+1}), (x_{n+2},\dots ,  x_{n+u+1}))$ with respect to the Haar measure on $\mathcal{F}_n^n\times\mathcal{F}_n^{\Gamma_1}\times\cdots\times\mathcal{F}_n^{\Gamma_{u+1}}$, the group $\mathcal{F}_n/[Y(S_1), S_2]$ lies in $A$. 
 \end{definition}
 This measure is invariant under conjugating $\Gamma_i$ for each $i$, so we can think of $\mu_{n,\underline{\Gamma}}$ as being associated to a $u+1$-tuple of conjugacy classes of subgroups of $\Gamma$.  

 Next we give the analogue of Theorem $4.12$ of \cite{lwzb} and Theorem $6.3(1)$ of \cite{LW2}. See the beginning of Section $4$ of \cite{lwzb} for definitions of $\mathcal{A}_H$, $\mathcal{N}$, and the lines before Theorem $4.12$ of \cite{lwzb} for the definition of $\lambda(\mathcal{C},H,G)$ (see also \eqref{eq: lambda}) appearing in the next theorem.
 \begin{theorem}\label{thm: H-probability}
    Let $\mathcal{C}$ be a finite set of finite $|\Gamma|'$-$\Gamma$-groups, and let $H$ be a finite admissible $|\Gamma|'$-$\Gamma$-group of level $\mathcal{C}$. Then 
    \begin{eqnarray}\Prob\left(X_{n,\underline{\Gamma}}^\mathcal{C}\cong_\Gamma H\right)&=&\frac{|\Sur_{\Gamma}(\mathcal{F}_n, H)||H^\Gamma|^{n+1}}{|\Aut_\Gamma(H)||H|^n\prod_{i=1}^{u+1}|H^{\Gamma_i}|}\prod_{G\in\mathcal{A}_H}\prod_{j=0}^{m(\mathcal{C},n,H,G)-1}\left(1-\frac{h_{H\rtimes\Gamma}(G)^{j}|G^\Gamma|^{n+1}}{|G|^n\prod_{i=1}^{u+1}|G^{\Gamma_i}|}\right)\notag\\&&\cdot\prod_{G\in\mathcal{N}}\left(1-\frac{|G^\Gamma|^{n+1}}{|G|^n|\prod_{i=1}^{u+1}|G^{\Gamma_i}|}\right)^{m(\mathcal{C},n,H,G)}.\end{eqnarray}
    Taking $n\rightarrow\infty$, \begin{eqnarray}\lim_{n\rightarrow\infty}\Prob\left(X_{n,\underline{\Gamma}}^\mathcal{C}\cong_\Gamma H\right)&=&\frac{|H^\Gamma|}{|\Aut_\Gamma(H)|\prod_{i=1}^{u+1}|H^{\Gamma_i}|}\prod_{G\in\mathcal{A}_H}\prod_{j=1}^{\infty}\left(1-\lambda(\mathcal{C},H,G)\frac{h_{H\rtimes\Gamma}(G)^{-j}|G^\Gamma|}{\prod_{i=1}^{u+1}|G^{\Gamma_i}|}\right)\notag\\&&\cdot\prod_{G\in\mathcal{N}}e^{-\frac{|G^\Gamma|\lambda(\mathcal{C},H,G)}{\prod_{i=1}^{u+1}|G^{\Gamma_i}|}}, \end{eqnarray}
    unless $H$ is not of the form $X_{n,\underline{\Gamma}}^\mathcal{C}$ for any $n$, in which case the limit is $0$. 
 \end{theorem}

 \begin{proof}
    The proof is very similar to Theorem $4.12$ of \cite{lwzb} and Theorem $6.3(1)$ of \cite{LW2}. We present the argument here for completeness, highlighting minor differences. 

    Set $\overline{x_i}$ to be the image of $x_i\in \mathcal{F}_n$ in the quotient $\mathcal{F}_n^\mathcal{C}$. We have 
    \[\Prob\left(X_{n,\underline{\Gamma}}^\mathcal{C}\cong H\right)=\sum_{\substack{N\subset\mathcal{F}_n^{\mathcal{C}},\text{ closed subgroup}\\\mathcal{F}_n^{\mathcal{C}}\rtimes\Gamma\text{-normal, } \mathcal{F}_n^{\mathcal{C}}/N\cong H}}\Prob \left([Y(\{\overline{x_1},\dots \overline{x_{n+1}}\}),\overline{x_{n+2}},\dots, \overline{x_{n+u+1}}]_{\mathcal{F}_n^\mathcal{C}\rtimes\Gamma}=N\right).\]

    The probability $\Prob \left([Y(\{\overline{x_1},\dots \overline{x_{n+1}}\}),\overline{x_{n+2}},\dots, \overline{x_{n+u+1}}]_{\mathcal{F}_n^{\mathcal{C}}\rtimes\Gamma}=N\right)$ is the same probability as if we pick $r_i:=\overline{x_i}\in\mathcal{F}_n^\mathcal{C}$ for $1\leq i \leq n$ and $r_{n+i}:=\overline{x_{n+i}}\in (\mathcal{F}_n^\mathcal{C})^{\Gamma_i}$ for $1\leq i\leq u+1$ independently and uniformly at Haar random and getting the same event. 
    
    Consider the fundamental exact sequence $1\rightarrow R\rightarrow F\rightarrow H\rightarrow 1$ as in the proof of \Cref{thm: main} induced from the surjection $\pi: \mathcal{F}_n^\mathcal{C}\rightarrow H$ by quotienting out $N=\ker\pi$ of this surjection and $\mathcal{F}_n^{\mathcal{C}}$ by the intersection of maximal $\mathcal{F}_n^{\mathcal{C}}\rtimes\Gamma$-normal subgroups of $N$, denoted $M$ (so $R=N/M, F=\mathcal{F}_n^{\mathcal{C}}/M$). We set $S_1=(r_1,\dots, r_{n+1}), S_2=(r_{n+2},\dots, r_{n+u+1})$, and $\overline{S_1}, \overline{S_2}$ be their images in $F$, respectively, throughout this discussion. Now, 
    \begin{equation}\label{eq: probability in fixed kernel}\Prob\left([Y(S_1), S_2]\subset N\right)=\frac{1}{|Y(H)|^n|Y(H^{\Gamma_1})|\prod_{i=2}^{u+1}|H^{\Gamma_i}|}=\frac{|H^{\Gamma}|^{n+1}}{|H|^n\prod_{i=1}^u|H^{\Gamma_i}|}.\end{equation}

    In addition, we claim that 

    \begin{equation}\Prob\left([Y(S_1),S_2]=N\big|[Y(S_1), S_2]\subset N\right)=\Prob\left([Y(\overline{S_1}),\overline{S_2}]=R\right).
    \end{equation}
    This claim follows from the argument in Theorem $4.12$ of \cite{lwzb} and Theorem $6.3(1)$ of \cite{LW} that elements in $N$ $\mathcal{F}_n^\mathcal{C}\rtimes\Gamma$-normally generate $N$ if and only if their image in $N/M$ normally generate $N/M$, which gives
    \begin{eqnarray}
    &&\Prob\left([Y(S_1),S_2]=N\big|[Y(S_1), S_2]\subset N\right)\nonumber\\
    &=&\frac{|F^\Gamma|^{n+1}|M|^n\prod_{i=1}^{u+1}|M^{\Gamma_i}|\#\{(Y(\overline{S_1}), \overline{S_2})\text{ satisfying }[Y(\overline{S_1}),\overline{S_2}]=R \}}{|(\mathcal{F}_n^\mathcal{C})^{\Gamma}|^{n+1}|Y(N)|^n|Y(N^{\Gamma_1})|\prod_{i=2}^{u+1}|N^{\Gamma_i}|}\nonumber\\
    &=&\frac{\#\{(Y(\overline{S_1}),\overline{S_2})\text{ satisfying }[Y(\overline{S_1}),\overline{S_2}]=R \}}{|Y(R)|^n|Y(R^{\Gamma_1})|\prod_{i=2}^{u+1}|R^{\Gamma_i}|}\nonumber\\
    &=&\Prob\left([Y(\overline{S_1}),\overline{S_2}]=R\right)\nonumber
    \end{eqnarray}
    where $(\overline{S_1},\overline{S_2})$ is chosen independently and uniformly at random from $(R^n\times R^{\Gamma_1})\times(R^{\Gamma_2}\cdots \times R^{\Gamma_{u+1}})$. 
    
    We can calculate this last probability using \Cref{prop: relation probability}, and this quantity does not depend on $N$, so we can collect the sum to $|\Sur_{\Gamma}(\mathcal{F}_n,H)|/|\Aut_\Gamma(H)|$, giving us the first part of the proposition. 
    
    The second part of the proposition follows from similar limit calculations as that in Theorem $4.12$ of \cite{lwzb}:
    \[\lim_{n\rightarrow\infty}\left(1-\frac{|G^\Gamma|}{|Y(G)|^n\prod_{i=1}^{u+1}|G^{\Gamma_i}|}\right)^{m(\mathcal{C},n,H,G)}=e^{-\frac{|G^\Gamma|\lambda(\mathcal{C},H,G)}{\prod_{i=1}^{u+1}|G^{\Gamma_i}|}},\]
    in the case when $G$ is a nonabelian element of $\mathcal{N}$ discussed in \cite{lwzb}, and
    \[\lim_{n\rightarrow \infty}\prod_{j=0}^{m(\mathcal{C},n,H,G)-1}\left(1-\frac{h_{H\rtimes\Gamma}(G)^j|G^\Gamma|^{n+1}}{|G|^n\prod_{i=1}^{u+1}|G^{\Gamma_i}|}\right)=\prod_{j=1}^\infty\left(1-\lambda(\mathcal{C},H,G)\frac{h_{H\rtimes\Gamma}(G)^{-j}|G^\Gamma|}{\prod_{i=1}^{u+1}|G^{\Gamma_i}|}\right)\]
    when $G\in \mathcal{A}_H$. These limits only hold when the multiplicity approaches infinity as $n\rightarrow\infty$, which is equivalent to saying that for some $n$, $H$ is of the form $X_{n,\underline{\Gamma}}^\mathcal{C}$. Finally, we use $(4.18)$ of \cite{lwzb} to obtain $\lim_{n\rightarrow \infty}\frac{|\Sur_{\Gamma}(\mathcal{F}_n,H)|}{|Y(H)|^n}=1$, which finishes the proof.  
 \end{proof}

 \begin{remark}\label{rmk: probability 0}
    The limit of probabilities $\Prob\left(X_{n,\underline{\Gamma}}^\mathcal{C}\cong_\Gamma H\right)$ is $0$ if and only if for large enough $n$, $\Prob\left(X_{n,\underline{\Gamma}}^\mathcal{C}\cong_\Gamma H\right)=0$ if and only if $H$ is not of the form $X_{n,\underline{\Gamma}}^\mathcal{C}$ for any $n$. This is also equivalent to the existence of $G\in\mathcal{A}_H\cup\mathcal{N}$ so that $G=G^\Gamma$ and $m(\mathcal{C},n,H,G)>0$ for some $n$. 
 \end{remark}
 \begin{definition}\label{def: mu}
    Let $\mathcal{A}$ be an algebra of sets generated by basic open sets $U_{\mathcal{C},H}$, for finite admissible $|\Gamma|'$-$\Gamma$-groups $H$ and finite sets of finite $|\Gamma|'$-$\Gamma$-groups $\mathcal{C}$. For $A\in\mathcal{A}$, define 
    \begin{equation}\mu_{\underline{\Gamma}}(A):=\lim_{n\rightarrow \infty}\mu_{n,\underline{\Gamma}}(A).\end{equation}
 \end{definition}
 
 The previous theorem shows this limit exists. We now want to mimic Section $5$ of \cite{lwzb} to show that $\mu_{\underline{\Gamma}}$ extends to a Borel probability measure. Denote by $P_{n,\underline{\Gamma}}$ to be the following:
 \begin{equation}P_{n,\underline{\Gamma}}(U_{\mathcal{C},H}):=\prod_{G\in\mathcal{A}_H}\prod_{j=0}^{m(\mathcal{C},n,H,G)-1}\left(1-\frac{h_{H\rtimes\Gamma}(G)^{j}|G^\Gamma|^{n+1}}{|G|^n\prod_{i=1}^{u+1}|G^{\Gamma_i}|}\right)\prod_{G\in\mathcal{N}}\left(1-\frac{|G^\Gamma|^{n+1}}{|G|^n|\prod_{i=1}^{u+1}|G^{\Gamma_i}|}\right)^{m(\mathcal{C},n,H,G)}.\end{equation}
 
 Our goal is to replace every instance of $P_{u,n}$ in Section $5$ of \cite{lwzb} with $P_{n,\underline{\Gamma}}$ defined above. In fact, when $u=0,1$, $P_{u,n}=P_{n,\underline{\Gamma}}$ where $\underline{\Gamma}=(\Gamma), (1)$ respectively, and these are both $1$-tuples. We could also extend this random group model to include a parameter that changes the number of admissible relators by a constant independent of $n$ (i.e. relators of the form $r^{-1}\gamma(r)$ for all $\gamma$) as we vary $r$ in some subgroup of $\mathcal{F}_n$ of the form $\mathcal{F}_n^{\Gamma_0}$ where $\Gamma_0\subset\Gamma$ is a subgroup; we can certainly compute the moments of such a random group model using the same method as well. However, we don't need this generality when conjecturing the distribution of $G_{\text{\O}}^T(K)^{\mathcal{C}}$, so we do not pursue it in this paper. 

 Set $\mathcal{C}_\ell$ to be the set of isomorphism classes of all $|\Gamma|'$-$\Gamma$-groups of order $\leq \ell$, as in Section 5 of \cite{lwzb}. We have the following analogue of Corollary $5.9$ in \cite{lwzb}, proved in exactly the same way as Lemma $9.5$ of \cite{LW}. 

 \begin{corollary}\label{cor: probability lower bound}
    Suppose $\ell>1$, $n\geq 1$ are integers and $\widetilde{H}$ is a $\Gamma$-group of level $\mathcal{C}_{\ell-1}$. Then there is a positive constant $c(\ell,\widetilde{H})$ not depending on $n$, so that for each $\Gamma$-group $H$ of level $\mathcal{C}_\ell$ satisfying $H^{\mathcal{C}_{\ell-1}}\cong \widetilde{H}$, we have $P_{n,\underline{\Gamma}}(U_{\mathcal{C}_\ell,H})\geq c(\ell,\widetilde{H})$ or $P_{n,\underline{\Gamma}}(U_{\mathcal{C}_\ell,H})=0$. 
 \end{corollary}
 \begin{proof}
    Suppose first that $G\in\mathcal{A}_H$. By \Cref{cor: fixed point number}, we see that $|Y(G)|=h_{H\rtimes\Gamma}(G)^{m}, |Y(G^{\Gamma_1})|=h_{H\rtimes\Gamma}(G)^{m'}, |G^{\Gamma_i}|=h_{H\rtimes\Gamma}(G)^{m_i}$, for some nonnegative integers $m,m', m_i$. The product \[\prod_{j=0}^{m(\mathcal{C}_\ell,n,G,H)-1}\left(1-\frac{h_{H\rtimes\Gamma}(G)^j}{|Y(G)|^n|Y(G^{\Gamma_1})|\prod_{i=2}^{u+1}|G^{\Gamma_i}|}\right)\] 
    being nonzero implies that 
    \begin{align*}
        1-\frac{h_{H\rtimes\Gamma}(G)^j}{|Y(G)|^n|Y(G^{\Gamma_1})|\prod_{i=2}^{u+1}|G^{\Gamma_i}|}&=1-\frac{h_{H\rtimes\Gamma}(G)^{mn+m'+\sum_{i=2}^{u+1}m_i}}{h_{H\rtimes\Gamma}(G)^{mn+m'+\sum_{i=2}^{u+1}m_i-j}|Y(G)|^n|Y(G^{\Gamma_1})|\prod_{i=2}^{u+1}|G^{\Gamma_i}|}\\
        &=1-\frac{1}{h_{H\rtimes\Gamma}(G)^{mn+m'+\sum_{i=2}^{u+1}m_i-j}})\\
        &\geq 1-\frac{1}{2^{mn+m'+\sum_{i=2}^{u+1}m_i-j}}
    \end{align*}
    where $mn+m'+\sum_{j=2}^{u+1}m_i-j$ is a positive integer for all $0\leq j \leq m(\mathcal{C}_\ell,n,G,H)-1$. Thus we obtain the same estimate as in Lemma $9.5$ of \cite{LW}:
    \begin{align*}
        \prod_{G\in\mathcal{A}_H}\prod_{j=0}^{m(\mathcal{C}_\ell,n,H,G)-1}\left(1-\frac{|G^\Gamma|^{n+1}h_{H\rtimes\Gamma}(G)^{j}}{|G|^n\prod_{i=1}^{u+1}|G^{\Gamma_i}|}\right)&\geq\prod_{\substack{G\in\mathcal{A}_H\\ m(\mathcal{C}_\ell,n,G,H)\neq 0\\ \text{some} \:n}}\prod_{j=1}^\infty\left(1-\frac{1}{2^j}\right)\\
        &\geq \left(\prod_{j=1}^\infty\left(1-\frac{1}{2^j}\right)\right)^{\displaystyle\sum_{(M,A)\in\mathcal{C}\mathcal{F}_\Gamma(\overline{\mathcal{C}_\ell})}|\Sur(\widetilde{H}\rtimes\Gamma,A)|}
    \end{align*}
    where in the last line we use Lemma $5.4$ of \cite{lwzb}. This quantity only depends on $\ell,\widetilde{H}$ and not on $n$ or $H$, and is finite by finiteness of chief factor pairs $\mathcal{C}\mathcal{F}_\Gamma(\overline{\mathcal{C}_\ell})$ from Corollary $5.3$ of \cite{lwzb}. See \cite{lwzb} Section $5$ for the definition of $\mathcal{C}\mathcal{F}_\Gamma(\mathcal{C})$. 

    We can do a similar computation as Lemma $9.5$ of \cite{LW} for $G\in\mathcal{N}$ to get the inequality 
    \[\prod_{G\in\mathcal{N}}\left(1-\frac{1}{|Y(G)|^n|Y(G^{\Gamma_1})|\prod_{i=2}^{u+1}|G^{\Gamma_i}|}\right)^{m(\mathcal{C}_\ell,n,G,H)}\geq \left(1-\frac{1}{2}\right)^{2\sum_{G\in\mathcal{N}}\frac{m(\mathcal{C}_\ell,n,G,H)}{|Y(G)|^n|Y(G^{\Gamma_1})|\prod_{i=2}^{u+1}|G^{\Gamma_i}|}}\]
    whenever this product is nonzero, and the exponent is bounded from above via the following:
    {
    \allowdisplaybreaks
    \begin{align*}
    \sum_{G\in\mathcal{N}}\frac{m(\mathcal{C}_\ell,n,G,H)}{|Y(G)|^n|Y(G^{\Gamma_1})|\prod_{i=2}^{u+1}|G^{\Gamma_i}|}&=
        \sum_{G\in\mathcal{N}}\frac{|G^\Gamma|}{\prod_{i=1}^{u+1}|G^{\Gamma_i}|}\sum_{(E,\pi)\in\mathcal{E}_{\mathcal{C}_\ell}(G,H)}\frac{\sum_{D\in\mathcal{E}_H^E}\nu(D,E)\left(\frac{|Y(D)|}{|Y(H)|}\right)^n}{|Y(G)|^n|\Aut_{\Gamma,H}(E,\pi)|}\\
    &\leq  \sum_{G\in\mathcal{N}}\frac{|G^\Gamma|}{\prod_{i=1}^{u+1}|G^{\Gamma_i}|}\sum_{(E,\pi)\in\mathcal{E}_{\mathcal{C}_\ell}(G,H)}\frac{1}{|\Aut_{\Gamma,H}(E,\pi)|}\\
    &\leq \sum_{G\in\mathcal{N}}\frac{|G^\Gamma|}{\prod_{i=1}^{u+1}|G^{\Gamma_i}|}\#\mathcal{E}_{\mathcal{C}_\ell}(G,H)\\
    &\leq \sum_{G\in\mathcal{N}}\frac{|G^\Gamma|}{\prod_{i=1}^{u+1}|G^{\Gamma_i}|}\left(\sum_{(G,A)\in\mathcal{C}\mathcal{F}_\Gamma(\overline{\mathcal{C}_\ell})}|\Sur(\widetilde{H}\rtimes\Gamma,A/\Inn(G))|\right)\\
    &=\sum_{(G,A)\in\mathcal{C}\mathcal{F}_\Gamma(\overline{\mathcal{C}_\ell})}\frac{|G^\Gamma||\Sur(\widetilde{H}\rtimes\Gamma,A/\Inn(G))|}{\prod_{i=1}^{u+1}|G^{\Gamma_i}|}.
    \end{align*}
    }
    The first inequality follows from Lemma $4.6$ of \cite{lwzb}:
    \[\sum_{D\in \mathcal{E}_H^E}\nu(D,E)\left(\frac{|Y(D)|}{|Y(H)|}\right)^n=|\Sur_{\Gamma}(\mathcal{F}_n^{\mathcal{C}}\rightarrow H,\pi)|\leq |Y(G)|^n.\] 
    The second inequality is the trivial bound. The third inequality follows from Lemma $5.4$ of \cite{lwzb}. We again leave the definitions of $\mathcal{E}_{\mathcal{C}_\ell}(G,H), \mathcal{E}_H^E, \nu(D,E), \Aut_{\Gamma,H}(E,\pi), \Sur_\Gamma(\mathcal{F}_n^\mathcal{C}\rightarrow H, \pi)$ to Sections $4$ and $5$ of \cite{lwzb}, and $\Inn(G)$ is the set of inner automorphisms of $G$. Since the sum has finitely many terms, each term being finite, we have an upper bound on the exponent, hence a lower bound on $\prod_{G\in\mathcal{N}}(1-\frac{1}{|Y(G)|^n|Y(G^{\Gamma_\infty})|})^{m(\mathcal{C}_\ell,n,G,H)}$ only depending on $\ell,\widetilde{H}$ when the product is nonzero. 
 \end{proof}

 Since we know that the limit of $P_{n,\underline{\Gamma}}(U_{\mathcal{C}_\ell,H})$ exists as $n\rightarrow\infty$, the bound above gives that for $n$ large enough, we always have $P_{n,\underline{\Gamma}}(U_{\mathcal{C}_\ell,H})\geq c(\ell,\widetilde{H})$ or $P_{n,\underline{\Gamma}}(U_{\mathcal{C}_\ell,H})=0$. We now state the analogue of Lemma $5.10$ of \cite{lwzb}. 
 \begin{lemma}\label{lem: switch limits}
    Suppose $f_n(H,\ell)$ are nonnegative functions for all positive integers $n,\ell$, and $H$ is a $\Gamma$-group of level $\mathcal{C}_\ell$, such that 
    \begin{enumerate}
        \item There are nonnegative functions $g_n(H,\ell)$ so that 
    \begin{equation}f_n(H,\ell)=g_n(H,\ell)P_{n,\underline{\Gamma}}(U_{\mathcal{C}_\ell,H})\:\forall n,\ell,H,\end{equation}
    \item For every choice of $H,\ell$, the limit $g(H,\ell)=\displaystyle\lim_{n\rightarrow \infty}g_n(H,\ell)$ exists, and $g_n(H,\ell)\leq g(H,\ell)$ for all $n\geq 1$,
    \item For any choice of $\widetilde{H}$, a $\Gamma$-group of level $\mathcal{C}_{\ell-1}$, we have
    \begin{equation}\sum_{\substack{H\in\overline{\mathcal{C}_\ell}\\ H^{\mathcal{C}_{\ell-1}}\cong_{\Gamma} \widetilde{H}}}f_n(H,\ell)=f_n(\widetilde{H},\ell-1).\end{equation}
    \end{enumerate}
    Then $\displaystyle\lim_{n\rightarrow \infty}f_n(H,\ell)$ exists, and 
    \begin{equation}\sum_{H\in\overline{\mathcal{C}_\ell}}\lim_{n\rightarrow \infty}f_n(H,\ell)=\lim_{n\rightarrow \infty}f_n(1,1).\end{equation}
 \end{lemma}
 The proof is essentially the same as that in \cite{lwzb}, where we use that $\displaystyle\lim_{n\rightarrow\infty}P_{n,\underline{\Gamma}}(U_{\mathcal{C}_\ell,H})$ exists by \Cref{thm: H-probability}, replace any instance of $P_{u,n}(U_{\mathcal{C}_\ell,H})$ with $P_{n,\underline{\Gamma}}(U_{\mathcal{C}_\ell,H})$ in the proof of Lemma $5.10$ in \cite{lwzb}, and we use \Cref{cor: probability lower bound} instead of Corollary $5.9$ of \cite{lwzb}. 

 As in Theorem $5.12$ of \cite{lwzb}, we now obtain that $\mu_{\underline{\Gamma}}$ is countably additive on $\mathcal{A}$ using \Cref{lem: switch limits}. 

 \begin{theorem}\label{thm: mu countably additive}
    $\mu_{\underline{\Gamma}}$ is countably additive on the algebra $\mathcal{A}$ defined in \Cref{def: mu}.
 \end{theorem}
 \begin{proof}
    We proceed in the exact same way as the proof of Theorem $5.12$ of \cite{lwzb} by setting \[f_n(H,\ell):=\mu_{n,\underline{\Gamma}}(U_{\mathcal{C}_\ell,H}), \quad g_n(H)=\frac{|\Sur_\Gamma(\mathcal{F}_n,H)|}{|\Aut_\Gamma(H)||Y(H)|^n|Y(H^{\Gamma_1})|\prod_{i=2}^{u+1}|H^{\Gamma_i}|},\] 
    which are all nonnegative functions. We also have 
    \[g_n(H)\leq \frac{1}{|\Aut_\Gamma(H)||Y(H^{\Gamma_1})|\prod_{i=2}^{u+1}|H^{\Gamma_i}|},\text{ and }\lim_{n\rightarrow \infty}g_n(H)=\frac{1}{|\Aut_\Gamma(H)||Y(H^{\Gamma_1})|\prod_{i=2}^{u+1}|H^{\Gamma_i}|}.\] 
    
    Finally, since $\mu_{n,\underline{\Gamma}}(U_{\mathcal{C}_\ell,H})=\Prob (X_{n,\underline{\Gamma}}\in U_{\mathcal{C}_\ell,H})$, the third condition of the above lemma is satisfied for our choice of $f_n$ by countable additivity of the probability measure on $\mathcal{F}_n^n\times\mathcal{F}_n^{\Gamma_1}\times\cdots\times\mathcal{F}_n^{\Gamma_{u+1}}$. Thus, by \Cref{lem: switch limits},
    \begin{equation}\label{eq: mass does not escape}
    \sum_{\substack{H\in\overline{\mathcal{C}_\ell}}}\mu_{\underline{\Gamma}}(U_{\mathcal{C}_\ell,H})=\sum_{\substack{H\in\overline{\mathcal{C}_\ell}}}\lim_{n\rightarrow \infty}\mu_{n,\underline{\Gamma}}(U_{\mathcal{C}_\ell,H})=\lim_{n\rightarrow \infty}\mu_{n,\underline{\Gamma}}(U_{\mathcal{C}_1,1})=1.\end{equation}
    
    We now use this result to proceed in a similar way to Corollary $9.7$ of \cite{LW}. Namely, fix $\ell$ and suppose that $B=\bigsqcup_{j}U_{\mathcal{C}_\ell,H_j}$ where $j$ ranges over an at most countable indexing set, $H_j$ level $\mathcal{C}_\ell$, and for $j\neq j'$, $U_{\mathcal{C}_\ell,H_j}\neq U_{\mathcal{C}_\ell,H_{j'}}$, forcing the two basic open sets to be disjoint. Also assume that $B\in\mathcal{A}$. We want to show that $\mu_{\underline{\Gamma}}(B)=\sum_{j}\mu_{\underline{\Gamma}}(U_{\mathcal{C}_\ell,H_j}).$ The set function $\mu_{\underline{\Gamma}}$ is finitely additive as $\mu_{n,\underline{\Gamma}}$ is finitely additive. Finite additivity is enough to show that for $\bigsqcup_{j=1}^MU_{\mathcal{C}_\ell,H_j}\subset B$, we have $\sum_{j=1}^M\mu_{\underline{\Gamma}}(U_{\mathcal{C}_\ell,H_j})=\mu_{\underline{\Gamma}}(\bigsqcup_{j=1}^MU_{\mathcal{C}_\ell,H_j})\leq\mu_{\underline{\Gamma}}(B).$ Since there are at most countably many $G_j$ of level $\mathcal{C}_\ell$ not among the $H_j$, we can enumerate them, and $B=\mathcal{P}-\bigsqcup_{j=1}^\infty U_{\mathcal{C}_\ell,G_j}\subset \mathcal{P}-\bigsqcup_{j=1}^M U_{\mathcal{C}_\ell,G_j}$. Thus, $\mu_{\underline{\Gamma}}(B)\leq \mu_{\underline{\Gamma}}(\mathcal{P}-\bigsqcup_{j=1}^M U_{\mathcal{C}_\ell,G_j})=1-\sum_{j=1}^M\mu_{\underline{\Gamma}}(U_{\mathcal{C}_\ell,G_j}).$ Taking $M\rightarrow\infty$ and using \eqref{eq: mass does not escape}, we get 
    \[\sum_{j=1}^\infty\mu_{\underline{\Gamma}}(U_{\mathcal{C}_\ell,H_j})\leq\mu_{\underline{\Gamma}}(B)\leq 1-\sum_{j=1}^\infty\mu_{\underline{\Gamma}}(U_{\mathcal{C}_\ell,G_j})=\sum_{j=1}^\infty\mu_{\underline{\Gamma}}(U_{\mathcal{C}_\ell,H_j})\]
    which is the desired claim. 

    We recall that our goal is to show that $\mu_{\underline{\Gamma}}$ is countably additive on the algebra $\mathcal{A}$. We proceed in the same way as the proof of Theorem $9.1$ of \cite{LW}, except we replace every instance of $S_\ell$ with $\mathcal{C}_\ell$. 
 \end{proof}

 Thus, by the same argument as in Section $5.3$ of \cite{lwzb} we get that $\mu_{\underline{\Gamma}}$ can be uniquely extended to a Borel measure on $\mathcal{P}$. Every open set $U\subset\mathcal{P}$ is given by a countable disjoint union of basic open sets of the form $U_{\mathcal{C}_\ell,H_{j,\ell}}$, where $H_{j,\ell}$ is a level $\mathcal{C}_\ell$ finite admissible $|\Gamma|'$-$\Gamma$-group for some $j\in J_\ell$, an indexing set depending on $\ell$. Thus,
 \begin{align*}\mu_{\underline{\Gamma}}(U)=\mu_{\underline{\Gamma}}(\bigsqcup_{j,\ell}U_{\mathcal{C}_\ell,H_{j,\ell}})&=\sum_{j,\ell}\mu_{\underline{\Gamma}}(U_{\mathcal{C}_\ell,H_{j,\ell}})\\
 &=\sum_{j,\ell}\lim_{n\rightarrow\infty}\mu_{n,\underline{\Gamma}}(U_{\mathcal{C}_\ell,H_{j,\ell}})\\
 &\leq \liminf_{n\rightarrow\infty}\sum_{j,\ell}\mu_{n,\underline{\Gamma}}(U_{\mathcal{C}_\ell,H_{j,\ell}})\\
 &=\liminf_{n\rightarrow\infty}\mu_{n,\underline{\Gamma}}(U).\end{align*}
 
 Together with the metrizability of $\mathcal{P}$, we obtain by the Portmanteau theorem that the sequence of positive probability measures $\mu_{n,\underline{\Gamma}}$ converge weakly to the Borel measure extending $\mu_{\underline{\Gamma}}$ at basic opens $U_{\mathcal{C}_\ell,H}$. Note that the indicator function of $U_{\mathcal{C},H}$ is continuous for $\mathcal{C}$ finite. To reiterate, as long as $\mu_{\underline{\Gamma}}(U_{\mathcal{C},H})\neq 0$, we have 
 \begin{eqnarray}\mu_{\underline{\Gamma}}(U_{\mathcal{C},H})&=&\frac{1}{|\Aut_\Gamma(H)||Y(H^{\Gamma_1})|\prod_{i=2}^{u+1}|H^{\Gamma_i}|}\prod_{G\in\mathcal{A}_H}\prod_{j=1}^{\infty}\left(1-\lambda(\mathcal{C},H,G)\frac{h_{H\rtimes\Gamma}(G)^{-j}}{|Y(G^{\Gamma_1})|\prod_{i=2}^{u+1}|G^{\Gamma_i}|}\right)\notag\\
 &&\cdot\prod_{G\in\mathcal{N}}e^{-\frac{\lambda(\mathcal{C},H,G)}{|Y(G^{\Gamma_1})|\prod_{i=2}^{u+1}|G^{\Gamma_i}|}}.\end{eqnarray}

 We also want to compute the measure of $U_{\mathcal{C},H}$ when $\mathcal{C}$ is not necessarily finite, where we still keep the notation that $U_{\mathcal{C},H}=\{X\in\mathcal{P}\:|\: X^{\mathcal{C}}\cong_\Gamma H\}$. To emphasize that $U_{\mathcal{C},H}$ is not necessarily open, we write $V_{\mathcal{C},H}=U_{\mathcal{C},H}$ instead. For $\mathcal{C}_i$ an increasing sequence of finite sets of $|\Gamma|'$-$\Gamma$-groups so that $\bigcup_i\mathcal{C}_i=\mathcal{C}$, and $H$ level $\mathcal{C}$, we have $\displaystyle H\cong\varprojlim_{i}H^{\mathcal{C}_i}$ as $\Gamma$-groups. This gives 
 \[V_{\mathcal{C},H}=\bigcap_{i}U_{\mathcal{C}_i,H^{\mathcal{C}_i}}.\]
 Note that $U_{\mathcal{C}_i,H^{\mathcal{C}_i}}\subset U_{\mathcal{C}_{i-1},H^{\mathcal{C}_{i-1}}}$ for all $i$. Let $\lambda(\mathcal{C},H,G)$ (resp. $m(\mathcal{C},n,H,G)$) be defined as in Definition $5.14$ of \cite{lwzb} for possibly infinite $\mathcal{C}$ as the limit $\displaystyle\lim_{i\rightarrow \infty}\lambda(\mathcal{C}_i,H,G)$ (resp. $\displaystyle \lim_{i\rightarrow \infty}m(\mathcal{C}_i,n,H,G)$). For each $H\in\overline{\mathcal{C}}$, there is $i$ so that $H\in\overline{\mathcal{C}_i}$, and every $(E,\pi)\in\mathcal{E}_\mathcal{C}(H,G)$ is in $\mathcal{E}_{\mathcal{C}_i}(H,G)$ for some $i$. Therefore,  
 \begin{equation}\label{eq: lambda}\lambda(\mathcal{C},H,G)=\begin{cases}(h_{H\rtimes\Gamma}(G)-1)\displaystyle\sum_{(E,\pi)\in\mathcal{E}_\mathcal{C}(H,G)}\frac{1}{|\Aut_{\Gamma,H}(E,\pi)|} & \mbox{if } G\in\mathcal{A}_H\\ \displaystyle\sum_{(E,\pi)\in\mathcal{E}_\mathcal{C}(H,G)}\frac{1}{|\Aut_{\Gamma,H}(E,\pi)|} & \mbox{if } G\in\mathcal{N}.\end{cases}
 \end{equation}
 
 Thus, $\lambda(\mathcal{C},H,G)$ is independent of the limiting sequence $\mathcal{C}_i$ we take. Similarly, we can use Corollary $4.8$ of \cite{lwzb} to deduce that $m(\mathcal{C},n,H,G)$ is independent of the limiting sequence $\mathcal{C}_i$. 

 The following analogue of Theorem $5.15$ of \cite{lwzb} holds:
 \begin{theorem}\label{thm: measure closed set}
    Let $\mathcal{C}$ be an arbitrary set of $|\Gamma|'$-$\Gamma$-groups. Let $H$ be a level $\mathcal{C}$ finite group. Then either $\mu_{\underline{\Gamma}}(V_{\mathcal{C},H})=0$, or 
    \begin{eqnarray}\mu_{\underline{\Gamma}}(V_{\mathcal{C},H})&=&\frac{1}{|\Aut_\Gamma(H)||Y(H^{\Gamma_1})|\prod_{i=2}^{u+1}|H^{\Gamma_i}|}\prod_{G\in\mathcal{A}_H}\prod_{j=1}^{\infty}\left(1-\lambda(\mathcal{C},H,G)\frac{h_{H\rtimes\Gamma}(G)^{-j}}{|Y(G^{\Gamma_1})|\prod_{i=2}^{u+1}|G^{\Gamma_i}|}\right)\notag\\
 &&\cdot\prod_{G\in\mathcal{N}}e^{-\frac{\lambda(\mathcal{C},H,G)}{|Y(G^{\Gamma_1})|\prod_{i=2}^{u+1}|G^{\Gamma_i}|}}.\end{eqnarray}
 \end{theorem}
 \begin{proof}
    As stated below Theorem $5.15$ of \cite{lwzb}, this statement follows from the same argument as Lemma $11.3$ of \cite{LW}, with some constants changed and every instance of $T_m$ changed to $\mathcal{C}_m$, and every instance of $S$ changed to $\mathcal{C}$. Indeed, let $\mathcal{C}_m$ be an increasing sequence of finite sets of $|\Gamma|'$-$\Gamma$-groups so that $\bigcup_m\mathcal{C}_m=\mathcal{C}$. From $\mu_{\underline{\Gamma}}$ being a finite Borel measure,
    \[\mu_{\underline{\Gamma}}(V_{\mathcal{C},H})=\lim_{m\rightarrow\infty}\mu_{\underline{\Gamma}}(U_{\mathcal{C}_m,H^{\mathcal{C}_m}})=\lim_{m\rightarrow\infty}\lim_{n\rightarrow\infty}\mu_{n,\underline{\Gamma}}(U_{\mathcal{C}_m,H^{\mathcal{C}_m}}).\]
    
    If $|Y(G)|=1$ for some $G\in\mathcal{A}_H\cup\mathcal{N}$ then $m(\mathcal{C}_m,n,H,G)$ is constant in $n$, and constant in $m$ for $m$ large enough as well. If any $G\in\mathcal{A}_H\cup\mathcal{N}$ satisfies this condition with positive multiplicity, then $\mu_{\underline{\Gamma}}(U_{\mathcal{C}_m,H})=0$ for $m$ large enough, so $\mu_{\underline{\Gamma}}(V_{\mathcal{C},H})=0$. If this never happens for any $G\in\mathcal{A}_H\cup\mathcal{N}$, then the product representation of $\mu_{\underline{\Gamma}}(U_{\mathcal{C}_m,H})$ is valid. Also recall that $\lambda(\mathcal{C}_m,H,G)$ is increasing in $m$ as the summand is equal and positive, and the indexing set $\mathcal{E}_{\mathcal{C}_m}(H,G)$ can only become bigger. 

    Since $H$ is finite, it is of level $\mathcal{C}_m$ for some $m$. For such $m$, 
    \[\lim_{n\rightarrow\infty}\mu_{n,\underline{\Gamma}}(U_{\mathcal{C}_m,H^{\mathcal{C}_m}})=\mu_{\underline{\Gamma}}(U_{\mathcal{C}_m,H}).\]
    Now,
    \begin{equation}\label{eq: product terms}\prod_{j=1}^{\infty}\left(1-\lambda(\mathcal{C}_m,H,G)\frac{h_{H\rtimes\Gamma}(G)^{-j}}{|Y(G^{\Gamma_1})|\prod_{i=2}^{u+1}|G^{\Gamma_i}|}\right) \quad\text{and}\quad e^{-\frac{\lambda(\mathcal{C}_m,H,G)}{|Y(G^{\Gamma_1})|\prod_{i=2}^{u+1}|G^{\Gamma_i}|}}\end{equation}
    lie in the closed interval $[0,1]$ because they are limit of probabilities. If any term in the first product of \eqref{eq: product terms} is negative, there is a term that is equal to $0$ by the proof of \Cref{prop: relation probability}. Therefore, the terms in \eqref{eq: product terms} is nonincreasing in $m$. Using that $\lambda(\mathcal{C}_m,H,G)$ is eventually constant in $m$ and converges to $\lambda(\mathcal{C},H,G)$, we obtain 
    \begin{eqnarray}
    \mu_{\underline{\Gamma}}(U_{\mathcal{C},H})
    &=&\lim_{m\rightarrow\infty}\frac{1}{|\Aut_\Gamma(H)||Y(H^{\Gamma_1})|\prod_{i=2}^{u+1}|H^{\Gamma_i}|}\prod_{G\in\mathcal{A}_H}\prod_{j=1}^{\infty}\left(1-\frac{\lambda(\mathcal{C}_m,H,G)h_{H\rtimes\Gamma}(G)^{-j}}{|Y(G^{\Gamma_1})|\prod_{i=2}^{u+1}|G^{\Gamma_i}|}\right)\nonumber \\&&\cdot\prod_{G\in\mathcal{N}}e^{-\frac{\lambda(\mathcal{C}_m,H,G)}{|Y(G^{\Gamma_1})|\prod_{i=2}^{u+1}|G^{\Gamma_i}|}}\nonumber\\
    &=&\frac{1}{|\Aut_\Gamma(H)||Y(H^{\Gamma_1})|\prod_{i=2}^{u+1}|H^{\Gamma_i}|}\prod_{G\in\mathcal{A}_H}\lim_{m\rightarrow\infty}\prod_{j=1}^{\infty}\left(1-\frac{\lambda(\mathcal{C}_m,H,G)h_{H\rtimes\Gamma}(G)^{-j}}{|Y(G^{\Gamma_1})|\prod_{i=2}^{u+1}|G^{\Gamma_i}|}\right)\nonumber\\&&\cdot\prod_{G\in\mathcal{N}}\lim_{m\rightarrow\infty}e^{-\frac{\lambda(\mathcal{C}_m,H,G)}{|Y(G^{\Gamma_1})|\prod_{i=2}^{u+1}|G^{\Gamma_i}|}}\nonumber\\
    &=&\frac{1}{|\Aut_\Gamma(H)||Y(H^{\Gamma_1})|\prod_{i=2}^{u+1}|H^{\Gamma_i}|}\prod_{G\in\mathcal{A}_H}\prod_{j=1}^{\infty}\left(1-\frac{\lambda(\mathcal{C},H,G)h_{H\rtimes\Gamma}(G)^{-j}}{|Y(G^{\Gamma_1})|\prod_{i=2}^{u+1}|G^{\Gamma_i}|}\right)\nonumber\\&&\cdot\prod_{G\in\mathcal{N}}e^{-\frac{\lambda(\mathcal{C},H,G)}{|Y(G^{\Gamma_1})|\prod_{i=2}^{u+1}|G^{\Gamma_i}|}}\nonumber.
    \end{eqnarray}
    
 \end{proof}

\section{\texorpdfstring{$H$-moments of the random group model}{H-moments of the random group model}}\label{sec: moment}
 In this final section, we compute the $H$-moments of the measure $\mu_{\underline{\Gamma}}$, closely following Section $6$ of \cite{lwzb} and using \Cref{lem: switch limits}. At the end of the section, we use a result of \cite{saw} to state the probability version of \Cref{conj: main} as well. 

 Denote by $\mathbb{E}_{n,\underline{\Gamma}}$ the expected value with respect to $\mu_{n,\underline{\Gamma}}$ and by $\mathbb{E}_{\underline{\Gamma}}$ the expected value with respect to $\mu_{\underline{\Gamma}}$. The idea is to first compute the limit of $\mathbb{E}_{n,\underline{\Gamma}}$, then show this limit agrees with $\mathbb{E}_{\underline{\Gamma}}$ when we compute moments with respect to the different measures. Note that Section $6$ of \cite{lwzb} uses the notation $A$ for our groups $H$ and vice versa for \Cref{thm: H-moment}. 

 \begin{lemma}\label{lem: limit H-moment mun}
    Let $H$ be a finite admissible $|\Gamma|'$-$\Gamma$-group. Then 
    \begin{equation}\lim_{n\rightarrow \infty}\mathbb{E}_{n,\underline{\Gamma}}(|\Sur_{\Gamma}(X,H)|)=\frac{|H^\Gamma|}{\prod_{i=1}^{u+1}|H^{\Gamma_i}|}.\end{equation}
 \end{lemma}

 \begin{proof}
    This lemma is the direct analogue of Lemma $6.1$ of \cite{lwzb}. Indeed, we have the equality
    \begin{equation}\label{eq: moments of En}\mathbb{E}_{n,\underline{\Gamma}}(|\Sur_\Gamma(X,H)|):=\int_{\mathcal{P}}|\Sur_\Gamma(X,H)|d\mu_{n,\underline{\Gamma}}=\frac{|\Sur_{\Gamma}(\mathcal{F}_n,H)|}{|Y(H)|^n|Y(H^{\Gamma_1})|\prod_{i=2}^{u+1}|H^{\Gamma_i}|}.\end{equation}
    This equality follows from decomposing the integrand into a sum of indicator functions
    \[|\Sur_\Gamma(X,H)|=\sum_{\phi:\mathcal{F}_n\twoheadrightarrow H}1_\phi(X),\]
    where $1_\phi(X)$ is $1$ if $\phi$ factors through $X$, and $0$ otherwise. This leads us to compute the probability that when we pick $(S_1,S_2)$ with respect to the Haar measure on $\mathcal{F}_n^n\times\mathcal{F}_n^{\Gamma_1}\times\cdots\times\mathcal{F}_n^{\Gamma_{u+1}}$ as in \Cref{def: measure mun}, the group $[Y(S_1),S_2]_{\mathcal{F}_n\rtimes\Gamma}$ is contained in the kernel of a fixed surjection $\phi:\mathcal{F}_n\twoheadrightarrow H$, which is \eqref{eq: probability in fixed kernel}.

    Using \eqref{eq: moments of En} we finish by taking $n\rightarrow\infty$ and using $(4.18)$ of \cite{lwzb} again.
 \end{proof}
 We are now ready to compute the $H$-moments of $\mu_{\underline{\Gamma}}$.
 \begin{theorem}\label{thm: H-moment}
    Let $H$ be a finite admissible $|\Gamma|'$-$\Gamma$-group. Then
    \begin{equation}\mathbb{E}_{\underline{\Gamma}}(|\Sur_{\Gamma}(X,H)|)=\frac{|H^\Gamma|}{\prod_{i=1}^{u+1}|H^{\Gamma_i}|}.\end{equation}
 \end{theorem}

 \begin{proof}
    Once again, this is very similar to Theorem $6.2$ of \cite{lwzb} with slightly more terms. The proof involves several key steps. For a level $\mathcal{C}_\ell$ finite $\Gamma$-group $A$, define $1_{X^{\mathcal{C}_\ell}\cong A}$ to be the indicator function that is $1$ when $X^{\mathcal{C}_\ell}\cong_\Gamma A$, and $0$ otherwise. 
    
    We first set \[f_n(A,\ell):=\mathbb{E}_n(|\Sur_\Gamma(X,H)|\cdot 1_{X^{\mathcal{C}_\ell}\cong A})=\sum_{\phi:\mathcal{F}_n\twoheadrightarrow H}\Prob \left([Y(S_1),S_2]\subset \ker\phi\mbox{ , }\mathcal{F}_n^{\mathcal{C}_\ell}/[Y(\overline{S_1}),\overline{S_2}]\cong A\right)\]
    where the sum ranges over all $\Gamma$-equivariant surjective maps $\mathcal{F}_n\rightarrow H$. The last equality follows from similar considerations as the previous lemma, where again, $(S_1,S_2)$ is a Haar random tuple of elements in $\mathcal{F}_n^n\times\mathcal{F}_n^{\Gamma_1}\times\cdots\times\mathcal{F}_n^{\Gamma_{u+1}}$, and $\overline{S_1},\overline{S_2}$ is their image in $\mathcal{F}_n^{\mathcal{C}_\ell}$ as in the proof of \Cref{thm: H-probability}. Let $\overline{\phi}$ be the induced surjection $\overline{\phi}:\mathcal{F}_n^{\mathcal{C}_\ell}\rightarrow H^{\mathcal{C}_\ell}$. We know that 
    \[\Prob \left([Y(S_1),S_2]\subset \ker\phi\big|[Y(\overline{S_1}),\overline{S_2}]\subset \ker\overline{\phi}\right)=\frac{|Y(H^{\mathcal{C}_\ell})|^n|Y({H^{\mathcal{C}_\ell}}^{\Gamma_1})|\prod_{i=2}^{u+1}|{H^{\mathcal{C}_\ell}}^{\Gamma_i}|}{|Y(H)|^n|Y(H^{\Gamma_1})|\prod_{i=2}^{u+1}|H^{\Gamma_i}|}\]
    by computing each of the probabilities $[Y(S_1),S_2]\subset\ker\phi,[Y(\overline{S_1}),\overline{S_2}]\subset\ker\overline{\phi}$ separately. Note that ${H^{\mathcal{C}_\ell}}^{\Gamma_i}$ is the set of $\Gamma_i$-fixed elements of $H^{\mathcal{C}_\ell}$. 

    Next, fix $y\in Y(\ker\overline\phi)^n\times Y((\ker\overline\phi)^{\Gamma_1})\times (\ker\overline\phi)^{\Gamma_2}\times\cdots\times (\ker\overline\phi)^{\Gamma_{u+1}}$. A key step in the proof of this theorem is to prove the following claim:
    \begin{equation}\label{eq: H-moment key equation}\Prob\left([Y(S_1),S_2]\subset\ker\phi\big|(Y(\overline{S_1}),\overline{S_2})=y\right)=\frac{|Y(H^{\mathcal{C}_\ell})|^n|Y({H^{\mathcal{C}_\ell}}^{\Gamma_1})|\prod_{i=2}^{u+1}|{H^{\mathcal{C}_\ell}}^{\Gamma_i}|}{|Y(H)|^n|Y(H^{\Gamma_1})|\prod_{i=2}^{u+1}|H^{\Gamma_i}|}.\end{equation}
    
    The proof of the claim also follows from basic probability theory of Haar measures on profinite groups and the fact that the natural map $\ker\phi\rightarrow \ker\overline{\phi}$ is surjective (see Lemma $5.1$ of \cite{Liu2}). 

    Third, we have
    \begin{eqnarray}
    &&\Prob \left([Y(S_1),S_2]\subset \ker\phi\big|[Y(\overline{S_1}),\overline{S_2}]\subset \ker\overline{\phi}\mbox{ , }\mathcal{F}_n^{\mathcal{C}_\ell}/[Y(\overline{S_1}),\overline{S_2}]_{\mathcal{F}_n^{\mathcal{C}_\ell}\rtimes\Gamma}\cong_\Gamma A\right)\nonumber \\
    &=&\frac{\Prob \left([Y(S_1),S_2]\subset \ker\phi\mbox{ and } [Y(\overline{S_1}),\overline{S_2}]\subset \ker\overline{\phi}\mbox{, }\mathcal{F}_n^{\mathcal{C}_\ell}/[Y(\overline{S_1}),\overline{S_2}]_{\mathcal{F}_n^{\mathcal{C}_\ell}\rtimes\Gamma}\cong_\Gamma A\right)}{\Prob\left([Y(\overline{S_1}),\overline{S_2}]\subset \ker\overline{\phi}\mbox{ , }\mathcal{F}_n^{\mathcal{C}_\ell}/[Y(\overline{S_1}),\overline{S_2}]_{\mathcal{F}_n^{\mathcal{C}_\ell}\rtimes\Gamma}\cong_\Gamma A\right)}\nonumber\\
    &=&\frac{\displaystyle\sum_y\Prob \left([Y(S_1),S_2]\subset \ker\phi, (Y(\overline{S_1}),\overline{S_2})=y\right)}{\displaystyle\sum_y\Prob \left((Y(\overline{S_1}),\overline{S_2})=y\right)}\nonumber\\
    &=&\frac{\displaystyle\sum_y\Prob \left([Y(S_1),S_2]\subset \ker\phi|(Y(\overline{S_1}),\overline{S_2})=y\right)}{\#\Big\{y\mbox{ such that }\mathcal{F}_n^{\mathcal{C}_\ell}/[y]_{\mathcal{F}_n^{\mathcal{C}_\ell}\rtimes\Gamma}\cong_\Gamma A\Big\}}\nonumber\\
    &=&\frac{|Y(H^{\mathcal{C}_\ell})|^n|Y({H^{\mathcal{C}_\ell}}^{\Gamma_1})|\prod_{i=2}^{u+1}|{H^{\mathcal{C}_\ell}}^{\Gamma_i}|}{|Y(H)|^n|Y(H^{\Gamma_1})|\prod_{i=2}^{u+1}|H^{\Gamma_i}|},\nonumber
    \end{eqnarray}
    where the sum in the third and fourth line ranges over all $y\in Y(\ker\overline\phi)^n\times Y((\ker\overline\phi)^{\Gamma_1})\times (\ker\overline\phi)^{\Gamma_2}\times\cdots\times (\ker\overline\phi)^{\Gamma_{u+1}}$ for which $\mathcal{F}_n^{\mathcal{C}_\ell}/[y]_{\mathcal{F}_n^{\mathcal{C}_\ell}\rtimes\Gamma}\cong A,$ and we use that the summand $P=\Prob ((Y(\overline{S_1}),\overline{S_2})=y)$ is independent of such $y$. The last line follows from \eqref{eq: H-moment key equation}. Note that if no such $y$ exists than the original quantity $f_n(A,\ell)$ is $0$ to begin with, so we can assume each of the sums have a nonzero number of terms. 
    
    Now we can compute $f_n(A,\ell)$ using conditional probability to be 
    \begin{eqnarray}
        &&f_n(A,\ell)\nonumber\\
        &=&\frac{|Y(H^{\mathcal{C}_\ell})|^n|Y({H^{\mathcal{C}_\ell}}^{\Gamma_1})|\prod_{i=2}^{u+1}|{H^{\mathcal{C}_\ell}}^{\Gamma_i}|}{|Y(H)|^n|Y(H^{\Gamma_1})|\prod_{i=2}^{u+1}|H^{\Gamma_i}|}\sum_{\phi:\mathcal{F}_n\twoheadrightarrow H}\Prob \left(\begin{gathered} \relax [Y(\overline{S_1}),\overline{S_2}]\subset \ker\overline{\phi}\mbox{ , }\\ \mathcal{F}_n^{\mathcal{C}_\ell}/[Y(\overline{S_1}),\overline{S_2}]_{\mathcal{F}_n^{\mathcal{C}_\ell}\rtimes\Gamma}\cong_\Gamma A\end{gathered}\right)\nonumber\\
        &=&\frac{|Y(H^{\mathcal{C}_\ell})|^n|Y({H^{\mathcal{C}_\ell}}^{\Gamma_1})|\prod_{i=2}^{u+1}|{H^{\mathcal{C}_\ell}}^{\Gamma_i}|}{|Y(H)|^n|Y(H^{\Gamma_1})|\prod_{i=2}^{u+1}|H^{\Gamma_i}|}\sum_{\phi:\mathcal{F}_n\twoheadrightarrow H}\Prob \left(X_{n,\underline{\Gamma}}^{\mathcal{C}_\ell}\cong_\Gamma A\mbox{ and }[Y(\overline{S_1}),\overline{S_2}]\subset \ker\overline{\phi}\right)\nonumber\\
        &=&\frac{|Y(H^{\mathcal{C}_\ell})|^n|Y({H^{\mathcal{C}_\ell}}^{\Gamma_1})|\prod_{i=2}^{u+1}|{H^{\mathcal{C}_\ell}}^{\Gamma_i}|}{|Y(H)|^n|Y(H^{\Gamma_1})|\prod_{i=2}^{u+1}|H^{\Gamma_i}|}\sum_{\phi:\mathcal{F}_n\twoheadrightarrow H}\sum_{\substack{N\subset\mathcal{F}_n^{\mathcal{C}_\ell},\\ \mathcal{F}_n^{\mathcal{C}_\ell}/N\cong A, \\ N\subset \ker\overline{\phi}}}\frac{P_{n,\underline{\Gamma}}(U_{\mathcal{C}_\ell,A})}{|Y(A)|^n|Y(A^{\Gamma_1})|\prod_{i=2}^{u+1}|A^{\Gamma_i}|}\nonumber\\
        &=&\frac{|Y(H^{\mathcal{C}_\ell})|^n|Y({H^{\mathcal{C}_\ell}}^{\Gamma_1})|\prod_{i=2}^{u+1}|{H^{\mathcal{C}_\ell}}^{\Gamma_i}|}{|Y(H)|^n|Y(H^{\Gamma_1})|\prod_{i=2}^{u+1}|H^{\Gamma_i}|}\nonumber\\
        &&\cdot\sum_{\phi:\mathcal{F}_n\twoheadrightarrow H}\frac{\#\Big\{(\tau,\pi)\big|\tau\in \Sur_\Gamma(\mathcal{F}_n^{\mathcal{C}_\ell},A),\pi\in \Sur_\Gamma(A,H^{\mathcal{C}_\ell}),\pi\circ\tau=\overline{\phi}\Big\}}{|\Aut_\Gamma(A)||Y(A)|^n|Y(A^{\Gamma_1})|\prod_{i=2}^{u+1}|A^{\Gamma_i}|}P_{n,\underline{\Gamma}}(U_{\mathcal{C}_\ell,A})\nonumber\\
        &=&\frac{|H^{\mathcal{C}_\ell}|^n\left(\prod_{i=1}^{u+1}|{H^{\mathcal{C}_\ell}}^{\Gamma_i}|\right)|\Sur_{\Gamma,H^{\mathcal{C}_\ell}}(\rho,H\twoheadrightarrow H^{\mathcal{C}_\ell})||\Sur_\Gamma(\mathcal{F}_n^{\mathcal{C}_\ell},A)||\Sur_\Gamma(A,H^{\mathcal{C}_\ell})|P_{n,\underline{\Gamma}}(U_{\mathcal{C}_\ell,A})}{{|H^{\mathcal{C}_\ell}}^\Gamma|^{n+1}|Y(H)|^n|Y(H^{\Gamma_1})|\left(\prod_{i=2}^{u+1}|H^{\Gamma_i}|\right)|\Aut_\Gamma(A)||Y(A)|^n|Y(A^{\Gamma_1})|\left(\prod_{i=2}^{u+1}|A^{\Gamma_i}|\right)},\nonumber
    \end{eqnarray}
    where the third line follows from the same argument as the first part of \Cref{thm: H-probability}, and the fourth line follows from the summand in the third line being independent of $N$, rephrasing conditions about $\Gamma$-normal subgroups in terms of surjective morphisms, and counting. The fifth line follows from exactly the same calculations in the proof of Theorem 6.2 of \cite{lwzb}, where $\rho$ is the composition of the pro-$\mathcal{C}_\ell$ completion of $\mathcal{F}_n$ with $\overline{\phi}$, and $\Sur_{\Gamma,H^{\mathcal{C}_\ell}}(\rho,H\twoheadrightarrow H^{\mathcal{C}_\ell})$ is the number of surjections $\mathcal{F}_n\rightarrow H$ that induce the surjection $\mathcal{F}_n^{\mathcal{C}_\ell}\rightarrow H^{\mathcal{C}_\ell}$. We note that $|\Sur_{\Gamma,H^{\mathcal{C}_\ell}}(\rho,H\twoheadrightarrow H^{\mathcal{C}_\ell})|$ only depends on $H, \mathcal{C}_\ell$ and not on the maps themselves. 
    
    Now set 
    \[g_n(A,\ell):=\frac{|H^{\mathcal{C}_\ell}|^n\left(\prod_{i=1}^{u+1}|{H^{\mathcal{C}_\ell}}^{\Gamma_i}|\right)|\Sur_{\Gamma,H^{\mathcal{C}_\ell}}(\rho,H\twoheadrightarrow H^{\mathcal{C}_\ell})||\Sur_\Gamma(\mathcal{F}_n^{\mathcal{C}_\ell},A)||\Sur_\Gamma(A,H^{\mathcal{C}_\ell})|}{|{H^{\mathcal{C}_\ell}}^\Gamma|^{n+1}|Y(H)|^n|Y(H^{\Gamma_1})|\left(\prod_{i=2}^{u+1}|H^{\Gamma_i}|\right)|\Aut_\Gamma(A)||Y(A)|^n|Y(A^{\Gamma_1})|\left(\prod_{i=2}^{u+1}|A^{\Gamma_i}|\right)},\]
    and 
    \[g(A,\ell)=\frac{|Y({H^{\mathcal{C}_\ell}}^{\Gamma_1})|\left(\prod_{i=2}^{u+1}|{H^{\mathcal{C}_\ell}}^{\Gamma_i}|\right)|\Sur_\Gamma(A,H^{\mathcal{C}_\ell})|}{|Y(H^{\Gamma_1})|\left(\prod_{i=2}^{u+1}|H^{\Gamma_i}|\right)|\Aut_\Gamma(A)||Y(A^{\Gamma_1})|\left(\prod_{i=2}^{u+1}|A^{\Gamma_i}|\right)}.\]
    Using that $|\Sur_{\Gamma,H^{\mathcal{C}_\ell}}(\rho,H\twoheadrightarrow H^{\mathcal{C}_\ell})|\leq |\Hom_{\Gamma,H^{\mathcal{C}_\ell}}(\rho,H\twoheadrightarrow H^{\mathcal{C}_\ell})|=\frac{|Y(H)|^n}{|Y(H^{\mathcal{C}_\ell})|^n}$ as in the proof of Theorem $6.2$ of \cite{lwzb} yields $g_n(A,\ell)\leq g(A,\ell), \:g_n(A,\ell)\rightarrow g(A,\ell)$ as $n\rightarrow\infty$, and by definition $f_n(A,\ell)=g_n(A,\ell)P_{n,\underline{\Gamma}}(U_{\mathcal{C}_\ell,A})$. From the definition of $f_n$ in terms of expected values, condition $(3)$ of \Cref{lem: switch limits} is satisfied as well. Therefore by \Cref{lem: switch limits},
    
    \[\sum_{A\in\overline{\mathcal{C}_\ell}}\lim_{n\rightarrow\infty}f_n(A,\ell)=\lim_{n\rightarrow\infty}f_n(1,1)=\lim_{n\rightarrow\infty}\mathbb{E}_n(|\Sur_\Gamma(X,H)|)=\frac{|H^{\Gamma}|}{\prod_{i=1}^{u+1}|H^{\Gamma_i}|}.
    \]
    
    Finally, taking $\ell$ large enough so that $H$ is level $\mathcal{C}_\ell$, 
    \begin{align*}
        \sum_{A\in\overline{\mathcal{C}_\ell}}\lim_{n\rightarrow\infty}f_n(A,\ell)&=\sum_{A\in\overline{\mathcal{C}_\ell}}\lim_{n\rightarrow\infty}g_n(A,\ell)P_{n,\underline{\Gamma}}(U_{\mathcal{C}_\ell,A})\\
        &=\sum_{A\in\overline{\mathcal{C}_\ell}}\lim_{n\rightarrow\infty}g(A,\ell)P_{n,\underline{\Gamma}}(U_{\mathcal{C}_\ell,A})\\
        &=\sum_{A\in\overline{\mathcal{C}_\ell}}\lim_{n\rightarrow\infty}\frac{|\Sur_\Gamma(A,H)|}{|\Aut_\Gamma(A)||Y(A^{\Gamma_1})|\left(\prod_{i=2}^{u+1}|A^{\Gamma_i}|\right)}P_{n,\underline{\Gamma}}(U_{\mathcal{C}_\ell,A})\\
        &=\sum_{A\in\overline{\mathcal{C}_\ell}}|\Sur_\Gamma(A,H)|\lim_{n\rightarrow\infty}\frac{1}{|\Aut_\Gamma(A)||Y(A^{\Gamma_1})|\left(\prod_{i=2}^{u+1}|A^{\Gamma_i}|\right)}P_{n,\underline{\Gamma}}(U_{\mathcal{C}_\ell,A})\\
        &=\sum_{A\in\overline{\mathcal{C}_\ell}}|\Sur_\Gamma(A,H)|\mu_{\underline{\Gamma}}(U_{\mathcal{C}_\ell,A})\\
        &=\mathbb{E}(|\Sur_\Gamma(X,H)|).
    \end{align*}
    
 \end{proof}

 \Cref{thm: H-probability}, \Cref{thm: measure closed set}, and \Cref{thm: H-moment} show that the random group model we constructed gives a probability distribution whose $H$-moments are $\frac{|H^\Gamma|}{\prod_{i=1}^{u+1}|H^{\Gamma_i}|}$. As long as $H$ is prime to $\Delta_Q:=|\Cl_T(Q)||\mu(Q)||\Gamma|$ ($\Delta_Q$ is also multiplied by $\text{char}(Q)$ if $Q$ is a function field), \Cref{conj: main} predicts that the $H$-moment of the Galois group of the maximal pro-prime-to-$|\Gamma|$ unramified extension of $K$ split completely at $T$, filtered by product of ramified primes, is equal to the $H$-moment of the probability distribution of the random group model we constructed, where $\Gamma_i$ are fixed decomposition groups of primes in $T$. 
 \begin{remark}\label{rmk: uniqueness of moments}
    Uniqueness of moments with respect to $\mu_{\underline{\Gamma}}$ follows from the existence of a random group model with the desired moments, from the work of Sawin (see Theorem 1.2 of \cite{saw}). Indeed the only condition we need to check is that the moments are of order $|H|^{O(1)}$, which follows from the explicit formula for the moment $\frac{|H^\Gamma|}{\prod_{i=1}^{u+1}|H^{\Gamma_i}|}$. In particular, if \Cref{conj: main} is true for all $H$ prime to $\Delta_Q$, then we have the probability version of these conjectures:
 \end{remark}

 \begin{theorem}[Moment version implies the Probability version]\label{thm: moment implies probability}
    Under the same assumptions of \Cref{conj: main}, fix a finite set of $\Delta_Q'$-$\Gamma$-groups, $\mathcal{C}$. Suppose \Cref{conj: main} holds for all finite level $\mathcal{C}$ groups $H$. If $\#E'_{\Gamma,(\Gamma_v)_{v\in T}}(Q)\neq0$, then for any finite level $\mathcal{C}$ group $H$,
    \begin{equation}\lim_{X\rightarrow\infty}\frac{\sum_{D\leq X}\sum_{K\in E'_{\Gamma,(\Gamma_v)_{v\in T}}(D,Q)}1_H(G_{\text{\normalfont\O}}^T(K)^{\mathcal{C}})}{\sum_{D\leq X}\#E'_{\Gamma,(\Gamma_v)_{v\in T}}(D,Q)}=\mu_{\underline{\Gamma}}(U_{\mathcal{C},H}).\end{equation}
    where $1_H(G)$ is the indicator function that equals $1$ if $H\cong_{\Gamma}G$ and $0$ otherwise, and $\underline{\Gamma}=(\Gamma_v)_{v\in T}$, ordered in some way. 
 \end{theorem}

 By the same argument, if \Cref{conj: E''} is true for all level $\mathcal{C}$ groups $H$, we obtain the probability version of \Cref{conj: E''} by an application of Theorem $1.2$ of \cite{saw}, and similarly for \Cref{conj: main 2}. 

\bibliographystyle{alpha}
\bibliography{References}

\end{document}